\documentclass[reqno,10pt]{amsart}

\usepackage{a4wide}
\usepackage{color}
\usepackage[T2A]{fontenc}
\usepackage[utf8]{inputenc}
\usepackage[english]{babel}
\numberwithin{equation}{section}

\usepackage{amsmath, amsfonts, amssymb, amsthm, mathtools, bm, bbm, esint, mathrsfs, nccmath, upgreek, thmtools}
\usepackage[backend=biber,style=numeric-comp,sortcites,giveninits,doi=false,isbn=false,url=false,eprint=false, maxbibnames=99]{biblatex}

 \usepackage{relsize}

\usepackage[colorlinks,citecolor=blue,linkcolor=red]{hyperref}
\usepackage[noabbrev,capitalise,nameinlink]{cleveref}
\hypersetup{colorlinks={true},linkcolor={red},citecolor=blue}
\usepackage{csquotes}

\makeatletter
\newcommand{\bg}{\bBigg@{0.5}}
\newcommand{\biggg}{\bBigg@{3}}
\newcommand{\Biggg}{\bBigg@{3.5}}
\newcommand{\bigggg}{\bBigg@{4}}

\makeatother

\usepackage{graphicx}
\usepackage{resmes}

\AtBeginBibliography{\footnotesize\sloppy}

\DeclareNameAlias{author}{family-given}

\theoremstyle{plain}
\newtheorem{theo}{Theorem}[section]
\newtheorem{lem}[theo]{Lemma}
\newtheorem{prop}[theo]{Proposition}

\theoremstyle{definition}
\newtheorem{defi}[theo]{Definition}
\newtheorem{rem}[theo]{Remark}



\usepackage{scalerel}

\newcommand{\1}{\mathbbm{1}}
\newcommand{\G}{\mathit{\Gamma}}
\newcommand{\Cur}{\mathrm{Cur}}

\newcommand{\GDis}{\operatorname{{\kern-0.1em}GDis}}
\newcommand{\ADis}{\operatorname{{\kern-0.1em}ADis}}
\newcommand{\Sel}{\operatorname{{\kern-0.1em}Sel}}
\newcommand{\Dual}{\operatorname{{\kern-0.1em}Dual}}

\newcommand{\X}{\mathsf{X}}

\newcommand{\Nat}{\mathbb{N}}
\newcommand{\Y}{\mathsf{Y}}
\newcommand{\Z}{\mathsf{Z}}
\newcommand{\dX}{\mathsf{d}_{\mathsf{X}}}
\newcommand{\dY}{\mathsf{d}_{\mathsf{Y}}}
\newcommand{\dZ}{\mathsf{d}_{\mathsf{Z}}}
\newcommand{\mX}{\mathfrak{m}_{\mathsf{X}}}

\newcommand{\D}{\mathrm{d}}

\newcommand{\Leb}{\mathrm{L}}

\newcommand{\ms}{\operatorname{{\kern-0.1em}\mathsf{ms}}}

\newcommand{\En}{\mathrm{E}}
\newcommand{\Mod}{\mathrm{Mod}}

\newcommand{\dom}{\mathrm{dom}}

\newcommand{\Mas}{\operatorname{{\kern-0.1em}\mathbf{M}}}
\newcommand{\Len}{\operatorname{{\kern-0.1em}Len}}

\newcommand{\R}{\mathbb{R}}

\newcommand{\LIP}{\mathrm{LIP}}
\DeclareMathOperator*{\esup}{ess\,sup}

\makeatletter
\let\save@mathaccent\mathaccent
\newcommand*\if@single[3]{%
  \setbox0\hbox{${\mathaccent"0362{#1}}^H$}%
  \setbox2\hbox{${\mathaccent"0362{\kern0pt#1}}^H$}%
  \ifdim\ht0=\ht2 #3\else #2\fi
  }
\newcommand*\rel@kern[1]{\kern#1\dimexpr\macc@kerna}
\newcommand*\widebar[1]{\@ifnextchar^{{\wide@bar{#1}{0}}}{\wide@bar{#1}{1}}}
\newcommand*\wide@bar[2]{\if@single{#1}{\wide@bar@{#1}{#2}{1}}{\wide@bar@{#1}{#2}{2}}}
\newcommand*\wide@bar@[3]{%
  \begingroup
  \def\mathaccent##1##2{%
    \let\mathaccent\save@mathaccent
    \if#32 \let\macc@nucleus\first@char \fi
    \setbox\z@\hbox{$\macc@style{\macc@nucleus}_{}$}%
    \setbox\tw@\hbox{$\macc@style{\macc@nucleus}{}_{}$}%
    \dimen@\wd\tw@
    \advance\dimen@-\wd\z@
    \divide\dimen@ 3
    \@tempdima\wd\tw@
    \advance\@tempdima-\scriptspace
    \divide\@tempdima 10
    \advance\dimen@-\@tempdima
    \ifdim\dimen@>\z@ \dimen@0pt\fi
    \rel@kern{0.6}\kern-\dimen@
    \if#31
      \overline{\rel@kern{-0.6}\kern\dimen@\macc@nucleus\rel@kern{0.4}\kern\dimen@}%
      \advance\dimen@0.4\dimexpr\macc@kerna
      \let\final@kern#2%
      \ifdim\dimen@<\z@ \let\final@kern1\fi
      \if\final@kern1 \kern-\dimen@\fi
    \else
      \overline{\rel@kern{-0.6}\kern\dimen@#1}%
    \fi
  }%
  \macc@depth\@ne
  \let\math@bgroup\@empty \let\math@egroup\macc@set@skewchar
  \mathsurround\z@ \frozen@everymath{\mathgroup\macc@group\relax}%
  \macc@set@skewchar\relax
  \let\mathaccentV\macc@nested@a
  \if#31
    \macc@nested@a\relax111{#1}%
  \else
    \def\gobble@till@marker##1\endmarker{}%
    \futurelet\first@char\gobble@till@marker#1\endmarker
    \ifcat\noexpand\first@char A\else
      \def\first@char{}%
    \fi
    \macc@nested@a\relax111{\first@char}%
  \fi
  \endgroup
}
\makeatother

\makeatother

\makeatletter
\renewcommand{\rightmark}{\expandafter\@secondoftwo\topmark{}{}}
\makeatother

\title[A relaxation to the post-composition approach]{A relaxation to the post-composition approach\\
to metric-valued Sobolev maps}

\author[R. D. Oleinik]{Roman D. Oleinik}

\address{{International School for Advanced Studies, Trieste, Italy}\newline
\indent {Moscow Institute of Physics and
Technology, Dolgoprudny, Russia}
}

\email{oleinik.r@phystech.edu}

\usepackage{tcolorbox}
\begin{document}
\begin{abstract}
We establish that the definition of Sobolev mappings from metric measure spaces to metric spaces relying on the post-composition approach of Ambrosio--Reshetnyak admits, under certain transparent assumptions on the spaces, an equivalent reformulation that is substantially weaker than the principal version. In detail, the corresponding definition requires maps of this kind to meet two successive tests in order to be called Sobolev: first a ``qualitative'' one and then a ``quantitative'' one. At the same time, our result allows the relevant conclusion to be drawn based solely on the ``qualitative'' part. Remarkably, the conditions on the spaces sufficient for this to hold fit entirely into the singular flavor of the subject. More explicitly, such a phenomenon emerges if either the source space is of finite Hausdorff dimension and satisfies the Sobolev-to-Luzin--Lipschitz property or the target space is of finite Hausdorff dimension. All this provides a significant generalization of several earlier achievements on the topic.
\end{abstract}
\maketitle
\allowdisplaybreaks

\tableofcontents

\newpage

\section{Introduction}\label{ss:Intr}
The classical Sobolev theory, as well as its bounded-variation counterpart, is known to deal primarily with real-valued functions defined on domains of Euclidean spaces. Alongside this, a central leitmotif of the last decades within this subject lies in extending the classes of both source and target spaces for maps between which a certain amount of the corresponding analysis is still possible. And the development outcomes for each of these components meaningfully differ.

On one hand, the most general type of source spaces considered in the described regard is represented by so-called metric measure spaces, which are modeled by a triple $(\X,\dX,\mX)$ consisting of a metric space $(\X,\dX)$ and a Borel measure $\mX$ thereon, typically satisfying a priori some mild extra conditions. As great pieces of the literature in this direction (where mostly real-valued functions are studied, sometimes in parallel with Banach-valued ones too), we suggest the monographs \cite{HKST15} and \cite{GP20}. Notably, in such an abstract setting, there exist quite many approaches to the theory known to this day, which propose different answers to the question of what real-valued functions on the space should be called Sobolev. And one of the most powerful results in the area states that certain of those (the most ``canonical'' ones in a sense) end up being completely equivalent, under some reasonable technical assumptions on the space at least. While this fact was originally established in \cite{AGS14} and \cite{AGS13}, its slightly more modern exposition is comprehensively presented in \cite{AILP24}. As the corresponding output of what is important for our discussion at this stage, let us record the following. Given a metric measure space $\X=(\X,\dX,\mX)$ and $p\in [1,+\infty)$, we can associate with $\X$ the homogeneous space $\mathrm{S}^{1,p}(\X)$ of Sobolev functions $\X\to \R$ (where the homogeneity means, and will mean throughout the manuscript, in the usual way that there are no a priori integrability requirements for functions in the space); moreover, given any $\phi\in \mathrm{S}^{1,p}(\X)$, there can be constructed the so-called minimal weak upper gradient $|\mathbf{D} \phi|_p$ (forming basically a certain naturally defined element of $\Leb^p_+(\X)$), which should be viewed as ``the modulus of the gradient'' of $\phi$ (actually depending on $p$ at this level of generality, as the notation suggests), even though the notion of gradient itself does not have a preexisting meaning in such a framework. As a useful side comment, all the same is true in the realm of functions of bounded variation as well (up to some appropriate modifications). For the sake of clarity, especially for the experienced reader, we highlight that the notions invoked above will be understood in our manuscript specifically in the Newtonian sense, exactly the one adopted within \cite{HKST15}. And the necessary further details on all these objects will be given in \cref{ss:FunCl}.

It turns out on the other hand that, for the role of a target space within the mentioned discourse, one can take an essentially arbitrary metric space. And similarly to the situation with source spaces, there exist various approaches to the question of what metric-valued maps on a given space should be designated as Sobolev, the choice between which is moreover somewhat independent of the analogous choice discussed above. For example, several of those are explored, in the setting of maps of bounded variation though, within \cite{BNP25}; some additional information on this topic can also be found in \cite{GT21}, \cite{O25}, and \cite{O26.1}. And the very approach that falls within the scope of our work, the weakest (or, from another perspective, the broadest) among all in a sense, can be labeled as the post-composition one.

The post-composition approach appeared originally, in the context of metric-valued maps of bounded variation on Euclidean domains, in the work \cite{A90} by L. Ambrosio. Essentially the same concept but in application to metric-valued Sobolev maps on Euclidean domains was then considered by Y. G. Reshetnyak in \cite{R97}. The associated core idea underneath relies on the following natural principle: given a metric-valued map on a Euclidean domain, its ``modulus of the gradient'' must dominate the modulus of the gradient (now existing) of the post-composition of the map with any real-valued $1$-Lipschitz function on the target space. Thus, passing to the setting of metric measure spaces as source spaces, we immediately arrive at the following definition: given a metric measure space $\X=(\X,\dX,\mX)$, a metric space $\Y=(\Y,\dY)$, and $p\in [1,+\infty)$, the homogeneous space $\mathrm{S}^{1,p}(\X,\Y)$ of Sobolev maps $\X\to \Y$ is stipulated to be the collection of all maps $\Phi\colon \X\to \Y$ such that the conditions below hold.
\begin{itemize}
    \item[$\rm I)$] For any Lipschitz function $f\colon \Y\to \R$, one has
    \begin{equation}
        f\circ \Phi\in \mathrm{S}^{1,p}(\X).
    \end{equation}
    \item[$\rm II)$] There is $G\in \Leb^p_+(\X)$ such that, for any $1$-Lipschitz function $f\colon \Y\to \R$, one has
    \begin{equation}
        \big|\mathbf{D}(f\circ \Phi)\big|_p(x)\leq G(x)\qquad \text{for $\mX$-a.e. $x\in \X$}.
    \end{equation}
\end{itemize}
Notably, such a general formulation requires no regularity either from the source space or from the target space.

As should be detected from the previous definition, the post-composition approach is inherently divided into two separate levels: first, within the ``qualitative'' part, one requires that all the post-compositions of a map with real-valued Lipschitz functions are ``regular''; second, within the ``quantitative'' part, one requires that all such post-compositions are ``equi-regular''. In this regard, it seems reasonable to label corresponding maps that satisfy only the condition of the first level as ``weakly regular'', by a natural analogy with the terminology adopted in the setting of topological vector spaces. Thus, under the same entries as above, we can consider the homogeneous space $\mathrm{wS}^{1,p}(\X,\Y)$ of weakly Sobolev maps $\X\to \Y$ consisting of all such maps for which condition I) holds. And, by the very definition, we have the straightforward inclusion $\mathrm{S}^{1,p}(\X,\Y)\subseteq \mathrm{wS}^{1,p}(\X,\Y)$, which can be read in words as follows: any Sobolev map is weakly Sobolev. The necessary details on these Sobolev classes will be given in \cref{ss:FunCl}.

In certain situations though, it may happen that (for any map between given spaces) there is no need for the second level of the post-composition approach, as the first level already implies it, i.e., any weakly Sobolev map is Sobolev. Indeed, if the target metric space is just a finite-dimensional Banach space, or, at least, can be embedded in a biLipschitz way into such a one, the stated attribute can be easily verified to be in place, actually for any source metric measure space. However, this condition is known to be quite restrictive, i.e., the class of geometrically interesting spaces having it is rather poor, so it is a meaningful question to ask under what assumptions on either the source or target spaces one can detect the same effect. For example, in a moderately less obvious way than in the biLipschitz case, the desired feature can be established to hold true (again for any source metric measure space) if the target metric space admits a Lipschitz BLD-embedding (where the abbreviation stands for ``bounded-length-distortion'') into a finite-dimensional Banach space, in the sense used in \cite{LD13}, which is already a significantly more general condition. And in fact, this very observation was exactly what eventually led to the appearance of the work \cite{O26.2} by the author of the current manuscript, the role of which in our exposition is to be separately discussed in the sequel, and what became the starting point for all the analysis developed in the present paper.

The first attempt to discover a phenomenon of the indicated kind in the setting of generic target spaces, i.e., without any special embeddability characteristics, was due to the author of the current manuscript and A. I. Tyulenev in \cite{OT24}. There the authors dealt specifically with maps from segments of the real line, understood as equipped with the natural metric and measure (which form perhaps the simplest metric measure spaces to consider) to essentially arbitrary metric spaces. The main result of that work can be briefly formulated as follows: a map from a segment of the real line to a separable, complete metric space is weakly Sobolev if and only if it is Sobolev, in the same sense as fixed previously. The corresponding proof, however, relied heavily on exactly the $1$-dimensionality of the source space and therefore did not admit a direct extension to higher-dimensional source spaces. Also, a similar result was almost simultaneously reached, in slightly different terms and with quite different techniques though, in \cite{B26} by V. I. Bakhtin. Subsequently, such an attribute was remarkably established for $2$-dimensional (as well as reproved for $1$-dimensional) Euclidean domains by S. V. Pavlov (see the forthcoming paper \cite{P26}), with much more involved methods exploiting the Euclidean structure of the source space even to a greater extent. Again, the relevant proof therein could not be transferred to higher dimensions in a straightforward way. And this became a borderline for the results in this direction at that moment.

In addition, stepping slightly aside from our exact discourse, let us provide the reader with the following supplementary information. The very same result by Oleinik--Tyulenev in fact covered also the case of continuous maps of bounded variation (from segments of the real line), with appropriate adjustments to the problem, but curiously did not cover the case of mere maps of bounded variation (from segments of the real line), as the techniques exploited there broke down without the additional continuity assumption (which is automatically present in the Sobolev context). And it was in fact discovered soon that this assumption is indeed inevitable for all this to behave in the same way. Namely, the subsequent work \cite{ST26} by D. M. Stolyarov and A. I. Tyulenev (see also the later work \cite{S26} by D. M. Stolyarov) demonstrated the existence of a map from a segment of the real line to, for instance, a separable Hilbert space that is weakly of bounded variation but not of bounded variation, still in the very same sense. Such an important circumstance is quite useful to compare with what happens within our forthcoming discussion.

In the current manuscript, we establish a substantial generalization of (certain parts of) several previous results, namely the ones by Oleinik--Tyulenev, Bakhtin, and Pavlov, within a completely singular setting, i.e., our consideration involves nothing related to the smooth context. More specifically, as for the source space, we take a metric measure space $\X=(\X,\dX,\mX)$, the precise meaning of which in our exposition will be clarified in \cref{ss:PrInf}, and, as for the target space, we take a metric space $\Y=(\Y,\dY)$, on which some extra peripheral conditions will be imposed explicitly whenever necessary. Then, the main structural assumptions on both these spaces will be expressed purely in terms of their Hausdorff dimension, understood in the conventional sense. Additionally, as one more attribute playing a role in our result, we will deal with the so-called Sobolev-to-Luzin--Lipschitz property of metric measure spaces, the meaning of which can be found in \cref{ss:FunCl}. In short, it stipulates that any real-valued Sobolev function on the space has the Luzin--Lipschitz property, i.e., loosely speaking, can be decomposed into a countable collection of Lipschitz functions. Some relevant discussion on these matters will be provided in the sequel.

In its most digestible version, our main result can be formulated as follows. A more general statement, from a slightly different perspective, is given by \cref{theo:GenTheo}.
\begin{theo}\label{theo:MainRes}
    Let $\X=(\X,\dX,\mX)$ be a metric measure space, let $\Y=(\Y,\dY)$ be a separable metric space, let $p\in [1,+\infty)$. Suppose either of the conditions below holds.
    \begin{enumerate}
        \item[$\rm I)$] The space $\X$ is of finite Hausdorff dimension and satisfies the $p$-Sobolev-to-Luzin--Lipschitz property.
    \item[$\rm II)$] The space $\Y$ is of finite Hausdorff dimension. 
        \end{enumerate}
    Then the following equality holds:
    \begin{equation}
    \begin{gathered}
       \mathrm{wS}^{1,p}(\X,\Y)=\mathrm{S}^{1,p}(\X,\Y).
        \end{gathered}
    \end{equation}
\end{theo}
\noindent We separately draw the reader's attention to the fact that the listed conditions are importantly not required to hold simultaneously.

Let us make the following comments on the theorem above. It is a well-known feature of Sobolev functions, and even of functions of bounded variation, on Euclidean domains (understood in the classical sense) that such functions indeed have the Luzin--Lipschitz property, and therefore all such domains satisfy the introduced Sobolev-to-Luzin--Lipschitz property. Since such domains are moreover of finite Hausdorff dimension in the obvious way, we immediately obtain that our result encompasses this class of source spaces, which is already a complete novelty in dimensions higher than two. Much more generally, if the source metric measure space admits certain versions of the so-called doubling and Poincar{\'e} conditions, then it is known to satisfy both the corresponding assumptions as well; for more on this, we refer the reader, for instance, to \cite[Section 3.3, Section 8.1, Section 8.2]{HKST15}. Since these two conditions are reputed to be enjoyed by a large class of geometrically interesting spaces, it is fair to say that this portion of our result is remarkably universal from the perspective of modern geometric analysis. At the same time, the other incarnation of our result, i.e., the one that deals with additional assumptions on the target space, is in turn closer in spirit to the earlier-discussed situation with maps valued in finite-dimensional Banach spaces (or in metric spaces possessing one of the embeddability characteristics invoked there) and forms therefore its far-reaching extension. Thus, what we observe here is that two seemingly independent effects (which, however, can in fact be linked, through a specific notion introduced in \cref{ss:FunCl}) force a given weakly Sobolev map to get automatically ``regularized'' up to being genuinely Sobolev. And in our current opinion, yet somewhat unjustified, this is essentially the best one can hope to achieve in this direction.

As a clarification for the reader, our proof of \cref{theo:MainRes} is based on two core aspects presented in the previous literature. The first aspect is embodied by the main result of the earlier-addressed work \cite{O26.2} by the author of the current manuscript, which, as is already announced, should be viewed exactly through the lens of the touched BLD-embeddability characteristic of metric spaces. In a nutshell, it establishes the existence, on any ``reasonable'' metric space of finite Hausdorff dimension once equipped with any ``reasonable'' measure, of a countable collection of special charts covering the whole space that are represented by certain vector-valued Lipschitz functions on the space being also ``uniformly nondegenerate'' in some specific sense. The necessary details in application to our discussion will be provided within \cref{theo:BLD}. The second aspect is adopted from the previously-mentioned work \cite{P26} by S. V. Pavlov. Its essence, from the perspective of our problem, can be summarized in the following strategy: consider all the post-compositions of a given map with all possible finite-dimensional vector-valued $1$-Lipschitz functions and prove that the energies (the meaning of which will be disclosed in \cref{ss:FunCl}) of these post-compositions are equi-bounded. The corresponding reasoning will be in use within \cref{theo:GenTheo}. All the rest, while being rather useful on its own and equally quite cumbersome, can be considered as a more technical part.
 
We conclude here with an important note. The idea to generalize the $1$-dimensional result by Oleinik--Tyulenev to a broader context was there right from the start and remained in the air for a long period of time. At some subsequent moments, both these authors were in a collective process of contemplating this subject, yet they stuck to somewhat distinct methods on their path, which led to more autonomous attempts of each one to reach the goal. In the middle of August 2026, the author of the current manuscript was informed by A. I. Tyulenev that the latter obtained a certain result in the relevant direction, which initiated the ensuing paper \cite{T26}. In contrast, the backbone of the analysis developed here took its final shape in an unrelated way at the end of August 2026, after a prolonged reflection on the corresponding matters and particularly in the wake of a talk by S. V. Pavlov exactly on his very work watched by the author of the current manuscript and followed by an insightful joint discussion on the topic; all this was then also communicated to A. I. Tyulenev. And while some pieces of the final results within each of these papers visually overlap, the two should be noted to be entirely different at a conceptual level, as well as to employ substantially different techniques in their core. Thus, this study is declared to be independent of the one by A. I. Tyulenev to the full extent.

\section{Preliminary information}\label{ss:PrInf}
Here we introduce some notation and terminology of a general kind necessary for our subsequent exposition.

We will adopt the following usual “indeterminacy” conventions: $(+\infty)\cdot 0=0\cdot (+\infty)=0$.

Given a set $S$, by $\mathcal{P}(S)$, we will denote the power set of $S$.

Given sets $S_0,S,S'$ with $S\subseteq S_0$ and a map $\kappa\colon S_0\to S'$, by $\kappa|_{S}$, we will mean the function $S\to S'$ obtained by taking the restriction of $\kappa$ to $S$.

Given a set $S$, by $\1_S$, we will mean ``the indicator function'' of $S$, which becomes a real function whenever the ambient set is understood.

By $\R$, $\overline{\R}_+$, and $\R_+$, we will denote the sets, respectively, of the reals, of the extended nonnegative reals, and of the nonnegative reals. All these sets will be assumed equipped with their usual topologies, and the set $\R$ will be also considered as a metric space under its usual metric.

By $\Nat$ and $\Nat_0$, we will denote the sets, respectively, of the positive integers and of the nonnegative integers.

Given $i,i'\in \Nat_0$, we put
\begin{equation}
    \overline{i,i'}\coloneqq [i,i']\cap \Nat_0.
\end{equation}

By $\mathcal{L}^1$, we will mean the standard Lebesgue measure on $\R$.

Let $M\in \Nat_0$. By $\R^M$, we will denote the $M$-dimensional real coordinate space. By $\R^M_{\infty}$ and $\R^M_1$, we will denote the metric spaces obtained by endowing $\R^M$ with, respectively, the $\infty$-norm and the $1$-norm. Here, we warn the reader about the following: sometimes a given map to $\R^M$ will need to be considered both as a map to $\R^M_{\infty}$ and as a map to $\R^M_1$, which obviously might be of some difference at the level of metric properties of the map; such ``changes of the target'' will always be declared explicitly to avoid any confusions.

Now we recall certain concepts from the analysis on metric spaces. So, let $\Z=(\Z,\dZ)$ be a metric space.

Let $S\subseteq \Z$. By $\dim_{\mathrm{H}}(S)$, we will denote the Hausdorff dimension of $S$. In what follows, we will tacitly exploit the following well-known facts about this notion: Lipschitz maps between metric spaces do not increase the Hausdorff dimension of sets; any subset of a metric space is contained in a Borel set therein of the same Hausdorff dimension. Also, we put
\begin{equation}
    \overline{\dim}_{\mathrm{H}}(S)\coloneqq \inf\Big\{d\in \Nat_0\bigm|\dim_{\mathrm{H}}(S)\leq d\Big\};
\end{equation}
such a quantity, which, whenever finite, is always a nonnegative integer and thereby will sometimes be more convenient for us than the ordinary Hausdorff dimension, can be informally called the upper Hausdorff dimension of a given set. With this, we immediately notice that
\begin{equation}
    \dim_{\mathrm{H}}(S)\leq \overline{\dim}_{\mathrm{H}}(S)\leq \dim_{\mathrm{H}}(S)+1,
\end{equation}
the fact of which should be kept in mind everywhere further.

By $\overline{\mathrm{B}}_+(\Z)$, we will mean the collection of all Borel functions $\Z\to \overline{\R}_+$. In turn, by $\mathrm{B}_+(\Z)$, we will mean the subcollection of $\overline{\mathrm{B}}_+(\Z)$ made of all $\rho\in \overline{\mathrm{B}}_+(\Z)$ such that $\rho(z)\in \R_+$ for every $z\in \Z$.

By a Borel measure on $\Z$, we will mean an outer measure $\zeta$ on $\Z$ such that any Borel set $Q\subseteq \Z$ is $\zeta$-measurable. By a Borel-regular measure on $\Z$, we will mean a Borel measure $\zeta$ on $\Z$ such that, for any $S\subseteq \Z$, there is a Borel set $Q\subseteq \Z$ with $S\subseteq Q$ and $\zeta(Q)=\zeta(S)$.

By a curve in $\Z$, we will mean a continuous map from a compact segment of $\R$ to $\Z$. Given a curve $\alpha$ in $\Z$, by $\dom(\alpha)$, we will denote its domain. Given a curve $\alpha$ in $\Z$, by $\Len(\alpha)$, we will denote its length. By $\Cur_{\mathrm{r}}(\Z)$, $\Cur_{\mathrm{ac}}(\Z)$, and $\Cur_{\mathrm{L}}(\Z)$, we will denote the families, respectively, of all rectifiable curves, of all absolutely continuous curves, and of all Lipschitz curves in $\Z$. In particular, we have the chain
\begin{equation}
    \Cur_{\mathrm{L}}(\Z)\subseteq \Cur_{\mathrm{ac}}(\Z)\subseteq \Cur_{\mathrm{r}}(\Z).
\end{equation}
For the details on these notions, we refer the reader to \cite[Section 5.1]{HKST15}.

Let us recall the following parameterization-invariance property of rectifiable curves: given $\alpha\in \Cur_{\mathrm{r}}(\Z)$, $a,b\in \R$ with $a\leq b$, and a nondecreasing, continuous function $\sigma\colon [a,b]\to \R$ with $\sigma\big([a,b]\big)\subseteq \dom(\alpha)$, we have that $\alpha\circ \sigma\in \Cur_{\mathrm{r}}(\Z)$.

We now recall the notion of metric speed of rectifiable curves. Let $\alpha\in \Cur_{\mathrm{r}}(\Z)$. By the metric speed of $\alpha$, we will mean the function $\ms_{\alpha}\colon \dom(\alpha)\to \R_+$ given for any $t\in \dom(\alpha)$ by
\begin{equation}\label{eq:MetrSp}
    \ms_{\alpha}(t)\coloneqq \begin{dcases}
        \lim\limits_{\substack{t'\to t\\t'\in \dom(\alpha)}} \frac{\dZ\big(\alpha(t),\alpha(t')\big)}{|t'-t|},\qquad\parbox{20em}{\center{the point $t$ is a limit point of $\dom(\alpha)$\\ and the limit exists in $\mathbb{R}_+$,}}\\
    0,\qquad\qquad\qquad\qquad\qquad\qquad\quad\;\;\;\text{otherwise}.
    \end{dcases}
\end{equation}
It is well known that this function is Lebesgue-measurable and that, for $\mathcal{L}^1$-a.e. $t\in \dom(\alpha)$, exactly the first alternative in \eqref{eq:MetrSp} is applied. These facts can be found in \cite[Theorem 4.4.8]{HKST15}.

Let $a,b\in \R$ be with $a\leq b$, let $\sigma\colon [a,b]\to \R$ be a nondecreasing function. By the derivative of $\sigma$, we will mean the function $\partial \sigma\colon [a,b]\to \R_+$ given for any $t\in [a,b]$ by
\begin{equation}\label{eq:MonFuncDer}
    (\partial \sigma)(t) \coloneqq \begin{dcases}
        \lim\limits_{\substack{t'\to t\\t'\in [a,b]}} \frac{\sigma(t')-\sigma(t)}{t'-t},\qquad \parbox{20em}{\center{the point $t$ is a limit point of $[a,b]$\\ and the limit exists in $\mathbb{R}_+$,}}\\
0,\qquad\qquad\qquad\qquad\qquad\quad\;\;\;\;\text{otherwise}.
    \end{dcases}
\end{equation}
It is a folklore fact that this function is Lebesgue-measurable and that, for $\mathcal{L}^1$-a.e. $t\in [a,b]$, exactly the first alternative in \eqref{eq:MonFuncDer} is applied. 

The following chain rule for the metric speed of curves is also well known.
\begin{prop}\label{prop:ChainRulCur}
    Let $\alpha \in \Cur_{\mathrm{L}}(\Z)$, let $a,b\in \R$ be with $a\leq b$, let $\sigma\colon [a,b]\to \R$ be a nondecreasing, continuous function with $\sigma\big([a,b]\big)\subseteq \dom(\alpha)$. Then the following properties hold:
    \begin{gather}
    \begin{gathered}
    (\partial \sigma)(s)=\ms_{\alpha\circ \sigma}(s)=0 \qquad \text{for $\mathcal{L}^1$-a.e. $s\in \sigma^{-1}(R)$}\\
    \text{for any Borel set $R\subseteq \R$ with $\mathcal{L}^1(R)=0$};
    \end{gathered}\\
        \ms_{\alpha\circ \sigma}(s)=\ms_{\alpha}\big(\sigma(s)\big)(\partial \sigma)(s)\qquad \text{for $\mathcal{L}^1$-a.e. $s\in [a,b]$}.
    \end{gather}
    \begin{proof}
        As the statement is standard, we omit the details.
    \end{proof}
\end{prop}

We prefer to explicitly state the fact about the existence of the natural parameterization of rectifiable curves.
\begin{prop}\label{prop:NatParam}
    Let $\alpha\in \Cur_{\mathrm{r}}(\Z)$. Then there exists a unique pair consisting of $\alpha_{\mathrm{n}}\in \Cur_{\mathrm{L}}(\Z)$ and of a nondecreasing, continuous function $\sigma_{\alpha}\colon \dom(\alpha)\to\R$ such that the following properties hold:
    \begin{gather}
\sigma_{\alpha}\big(\dom(\alpha)\big)=\dom(\alpha_{\mathrm{n}})=\big[0,\Len(\alpha)\big];\\
        \alpha=\alpha_{\mathrm{n}}\circ \sigma_{\alpha};\\
        \ms_{\alpha_{\mathrm{n}}}(s)=1\qquad \text{for $\mathcal{L}^1$-a.e. $s\in \dom(\alpha_{\mathrm{n}})$}.
    \end{gather}
    \begin{proof}
        The existence is established in \cite[Section 5.1]{HKST15}. The uniqueness, in turn, can be verified routinely on the basis of the stated properties.
    \end{proof}
\end{prop}

We now recall the notion of line integration. Given $\alpha\in \Cur_{\mathrm{r}}(\Z)$ and $\rho\in \overline{\mathrm{B}}_+(\Z)$, by the line integral of $\rho$ over $\alpha$, we will mean the quantity
\begin{equation}
    \int\limits_{\alpha} \rho\coloneqq \int\limits_{\dom(\alpha_{\mathrm{n}})} \rho\big(\alpha_{\mathrm{n}}(s)\big)\D \mathcal{L}^1(s).
\end{equation}
The following parametrization-invariance property of the line integral is well known. Given $\alpha\in \Cur_{\mathrm{ac}}(\Z)$ and $\rho\in \overline{\mathrm{B}}_+(\Z)$, one has
\begin{equation}\label{eq:ParInvLinInt}
    \int\limits_{\alpha} \rho=\int\limits_{\dom(\alpha)} \rho\big(\alpha(t)\big) \ms_{\alpha}(t)\D \mathcal{L}^1(t).
\end{equation}
For the necessary details on this matter, we send the reader to \cite[Section 5.1]{HKST15}.

The following notion will be of use later. Given $\mathit{A}\subseteq \Cur_{\mathrm{r}}(\Z)$, a function $\rho\in \overline{\mathrm{B}}_+(\Z)$ will be called admissible for $\mathit{A}$ if
\begin{equation}
    \int\limits_{\alpha}\rho\geq 1\qquad \text{for any $\alpha\in \mathit{A}$}.
\end{equation}

Let $\Y=(\Y,\dY)$ be a metric space. By $\LIP(\Y,\Z)$, we will denote the collection of all Lipschitz maps $\Y\to \Z$. By $\LIP_1(\Y,\Z)$, we will denote the collection of all $1$-Lipschitz maps $\Y\to \Z$. For brevity, we also put $\LIP(\Y)\coloneqq \LIP(\Y,\R)$ and $\LIP_1(\Y)\coloneqq \LIP_1(\Y,\R)$. Let us note that the collection $\LIP(\Y)$ forms a vector space under the standard operations.

Let $\Y=(\Y,\dY)$ be a metric space, let $\nu$ be a Borel measure on $\Y$. By the $\infty$-modulus on $\Y$ with respect to $\nu$, we will mean the set function $\Mod_{\infty}^{\nu}\colon \mathcal{P}\big(\Cur_{\mathrm{r}}(\Y)\big)\to \overline{\R}_+$ given for any $\mathit{B}\subseteq \Cur_{\mathrm{r}}(\Y)$ by
\begin{equation}
    \Mod_{\infty}^{\nu}(\mathit{B})\coloneqq \inf\limits_{\rho} \esup\limits_{\nu} (\rho),
\end{equation}
where the infimum is taken over all functions $\rho\in \overline{\mathrm{B}}_+(\Y)$ admissible for $\mathit{B}$. This set function is known to be an outer measure on $\Cur_{\mathrm{r}}(\Y)$. For more on this notion, see \cite{AHPCS18}.

The following simple property will be used later.
\begin{prop}\label{prop:ModEst}
    Let $\Y=(\Y,\dY)$ be a metric space, let $\nu$ be a Borel measure on $\Y$, let $\mathit{B}\subseteq \Cur_{\mathrm{r}}(\Y)$. Then one has
    \begin{equation}
        \Mod_{\infty}^{\nu}\bigg(\Big\{\beta\in \Cur_{\mathrm{r}}(\Y)\bigm| \beta_{\mathrm{n}}\in \mathit{B}\Big\}\bigg)\leq \Mod_{\infty}^{\nu}(\mathit{B}).
    \end{equation}
    \begin{proof}
        The statement can be derived directly from the relevant definitions.
    \end{proof}
\end{prop}

By a metric measure space, we will mean a triple $(\X,\dX,\mX)$ where $(\X,\dX)$ is a separable, locally complete metric space and $\mX$ is a locally finite Borel-regular measure thereon.

For the rest of the section, let $\X=(\X,\dX,\mX)$ be a fixed metric measure space.

Let $S$ be a set, let $\kappa\colon \X\to S$ be a map. By an $\mX$-representative of $\kappa$, we will mean any map $\kappa'\colon \X\to S$ such that
\begin{equation}
    \kappa'(x)=\kappa(x)\qquad \text{for $\mX$-a.e. $x\in \X$}.
\end{equation}
By $[\kappa]_{\mX}$, we will denote the set of all $\mX$-representatives of $\kappa$. Also, slightly abusing the terminology, we will speak about $\mX$-representatives of $[\kappa]_{\mX}$, by which we will clearly mean any element of $[\kappa]_{\mX}$. Moreover, in the usual way, we will treat the set $[\kappa]_{\mX}$ as an $\mX$-a.e. defined function on $\X$ valued in $S$.

Let $\mathsf{T}$ be a topological space. A map $\Phi\colon \X\to \mathsf{T}$ will be called $\mX$-measurable if the set $\Phi^{-1}(Q)$ is $\mX$-measurable for any Borel set $Q\subseteq \mathsf{T}$. By $\mathrm{M}(\X,\mathsf{T})$, we will mean the collection of all $\mX$-measurable maps $\X\to \mathsf{T}$. For brevity, we also put $\mathrm{M}(\X)\coloneqq \mathrm{M}(\X,\R)$. We note separately that there is no identification up to $\mX$-a.e. equality here.

We now introduce the following collections:
\begin{gather}
    \Leb(\X)\coloneqq \Big\{[u]_{\mX}\bigm| u\in \mathrm{M}(\X)\Big\};\\
    \overline{\Leb}_+(\X)\coloneqq \Big\{[u]_{\mX}\bigm| u\in \mathrm{M}(\X,\overline{\R}_+)\Big\}.
\end{gather}
It should be also noted that the collection $\Leb(\X)$ forms in the obvious way a vector space under the standard operations. Next, we let $\Leb_+(\X)$ be the subcollection of $\overline{\Leb}_+(\X)$ made of all $v\in \overline{\Leb}_+(\X)$ such that $v(x)\in \R_+$ for $\mX$-a.e. $x\in \X$. In addition, given $p\in [1,+\infty)$, by $\Leb_+^p(\X)$, we will mean the collection of all $v\in \Leb_+(\X)$ with
\begin{equation}
    \int\limits_{\X} v^p\D\mX<+\infty.
\end{equation}

In the sequel, we will exploit the following well-known fact:
\begin{equation}\label{eq:BorRepr}
   \text{given $\varrho\in \Leb_+(\X)$, there is $\rho\in \mathrm{B}_+(\X)$ with $\varrho=[\rho]_{\mX}$}.
\end{equation}
Its proof is provided in \cite[Proposition 3.3.23]{HKST15}.

At some moment later, we will need the standard concepts of join and meet of an arbitrary family of (equivalence classes up to a.e. equality of) nonnegative measurable functions. The corresponding existence statement, for the join specifically, we formulate below.
\begin{prop}
    Let $\mathscr{V}\subseteq \overline{\Leb}_+(\X)$ be a nonempty set. Then there exists a unique element $v^{\vee}\in \overline{\Leb}_+(\X)$, which will be called the join of $\mathscr{V}$ and which will be denoted by
    \begin{equation}
        \bigvee\limits_{v\in \mathscr{V}} v,
    \end{equation}
    satisfying the following properties:
    \begin{enumerate}
        \item[$\rm I)$] it holds, for any $v\in \mathscr{V}$, that
        \begin{equation}
            v(x)\leq v^{\vee}(x)\qquad \text{for $\mX$-a.e. $x\in \X$};
        \end{equation}
        \item[$\rm II)$] given any $v'\in \overline{\Leb}_+(\X)$ such that it holds, for any $v\in \mathscr{V}$, that
        \begin{equation}
            v(x)\leq v'(x)\qquad \text{for $\mX$-a.e. $x\in \X$},
        \end{equation}
        one has
        \begin{equation}
            v^{\vee}(x)\leq v'(x)\qquad \text{for $\mX$-a.e. $x\in \X$}.
        \end{equation}
    \end{enumerate}
    Moreover, there exists $(v_n)_{n\in \Nat}\subseteq \mathscr{V}$ such that
    \begin{equation}
        v^{\vee}(x)=\sup\limits_{n\in \Nat} v_n(x)\qquad \text{for $\mX$-a.e. $x\in \X$}.
    \end{equation}
    \begin{proof}
        We omit the routine details, as the proof is classical.
    \end{proof}
\end{prop}
\noindent To be explicit, given $\mathscr{V}\subseteq \overline{\Leb}_+(\X)$, its meet, whose properties can be determined by dualizing those given above, will be denoted by
\begin{equation}
    \bigwedge\limits_{v\in \mathscr{V}} v.
\end{equation}

\begin{rem}
    We also point out the following easy consequence of the proposition above: whenever the family $\mathscr{V}$ from the proposition is countable, its join and meet are nothing but, respectively, the pointwise supremum and infimum of the elements from $\mathscr{V}$. This will be tacitly used later on.
\end{rem}

Now we introduce the finite-exponent moduli. Given $p\in [1,+\infty)$, by the $p$-modulus on $\X$, we will mean the set function $\Mod_p^{\mX}\colon \mathcal{P}\big(\Cur_{\mathrm{r}}(\X)\big)\to \overline{\R}_+$ given for any $\G\subseteq \Cur_{\mathrm{r}}(\X)$ by
\begin{equation}
    \Mod_p^{\mX}(\G)\coloneqq \inf\limits_{\rho} \int\limits_{\X} \rho^p \D \mX,
\end{equation}
where the infimum is taken over all functions $\rho\in \overline{\mathrm{B}}_+(\X)$ admissible for $\G$. Such set functions are also known to be outer measures on $\Cur_{\mathrm{r}}(\X)$. We refer the reader to \cite[Section 5.2]{HKST15} for more on this.

We record the following standard property of the moduli.
\begin{prop}\label{prop:ModZeroNegSet}
    Let $E\subseteq\X$ be an $\mX$-measurable set with $\mX(E)=0$. Then it holds, for $\Mod_p^{\mX}$-a.e. $\gamma\in \Cur_{\mathrm{r}}(\X)$, that
\begin{equation}
    \mathcal{L}^1\big(\gamma_{\mathrm{n}}^{-1}(E)\big)=0.
\end{equation}
\begin{proof}
    The statement follows from \cite[Lemma 5.2.15]{HKST15}.
\end{proof}
\end{prop}

Given a metric space $\Y=(\Y,\dY)$, a map $\Phi\colon \X\to \Y$ will be called $\mX$-essentially separably valued if, for some $\mX$-measurable set $E\subseteq \X$ with $\mX(E)=0$, the image $\Phi\big(\X\backslash E\big)$ is separable as a subspace of $\Y$. We then note immediately in this regard that any map from $\X$ to a separable metric space is automatically $\mX$-essentially separably valued, the fact of which should be kept in mind later on.

The following folklore fact allows one to derive the measurability of a map from the measurability of its post-compositions with Lipschitz functions.
\begin{theo}\label{theo:Pettis}
    Let $\Y=(\Y,\dY)$ be a metric space, let $\Phi\colon \X\to \Y$ be an $\mX$-essentially separably valued map such that
    \begin{equation}
        f\circ \Phi\in \mathrm{M}(\X)\qquad \text{for any $f\in \LIP(\Y)$}.
    \end{equation}
    Then one has
    \begin{equation}
    \Phi\in \mathrm{M}(\X,\Y).
    \end{equation}
    \begin{proof}
        The statement stems quite easily from the Pettis measurability theorem, which can be found in \cite[Section 3.1]{HKST15}. Some explanations are given also in \cite[Lemma 2.3]{OT24}.
    \end{proof}
\end{theo}

\section{Functional classes}\label{ss:FunCl}
Here we discuss certain aspects of the abstract Sobolev theory necessary for our presentation.

In this section, a metric measure space $\X=(\X,\dX,\mX)$ is assumed fixed.

We start with the following auxiliary class of maps. It can be viewed as a singular analog of the class of functions absolutely continuous on lines from the classical Sobolev calculus; see \cite[Section 6.3]{HKST15} for more on this.
\begin{defi}\label{def:ACC}
Let $\Y=(\Y,\dY)$ be a metric space, let $p\in [1,+\infty)$. By $\mathrm{ACC}_p(\X,\Y)$, the space of maps absolutely continuous on $p$-a.e. curve, we will mean the collection of all maps $\Phi\colon \X\to \Y$ such that it holds, for $\Mod^{\mX}_p$-a.e. $\gamma\in \Cur_{\mathrm{r}}(\X)$, that $\Phi\circ \gamma_{\mathrm{n}}\in \Cur_{\mathrm{ac}}(\Y)$.
\end{defi}

\begin{rem}
    As a warning for the reader, we note that the above-introduced class is not stable under changes of maps on sets of zero measure. More explicitly, even if some map belongs to this class, it is not the case in general that its given representative does.
\end{rem}

Below we recall the classical concept of weak upper gradient of a map, in the Newtonian sense specifically; see \cite[Section 6.2]{HKST15} for the details. We exploit a slightly nonstandard approach to it, more convenient for our narrative, but both are in fact completely equivalent as verified in \cite[Proposition 6.3.3]{HKST15}.
\begin{defi}
Let $\Y=(\Y,\dY)$ be a metric space, let $p\in [1,+\infty)$. Let $\Phi\in \mathrm{ACC}_p(\X,\Y)$. By a $p$-weak upper gradient of $\Phi$, we will mean any function $G\in \overline{\mathrm{B}}_+(\X)$ such that it holds, for $\Mod_p^{\mX}$-a.e. $\gamma\in \Cur_{\mathrm{r}}(\X)$ with $\Phi\circ \gamma_{\mathrm{n}}\in \Cur_{\mathrm{ac}}(\Y)$, that
\begin{equation}
    \ms_{\Phi\circ \gamma_{\mathrm{n}}}(s)\leq G\big(\gamma_{\mathrm{n}}(s)\big)\qquad \text{for $\mathcal{L}^1$-a.e. $s\in \dom(\gamma_{\mathrm{n}})$}.
\end{equation}
\end{defi}

Now we define the class of Dirichlet--Sobolev maps, which will play a crucial role in what follows. We send the reader to \cite[Section 7.1]{HKST15} for a thorough exposition of the topic. Let us also note here that the definition below has nothing to do with the post-composition approach and, in a sense, is of more direct and intrinsic nature than it.
\begin{defi}
Let $\Y=(\Y,\dY)$ be a metric space, let $p\in [1,+\infty)$. By $\mathrm{D}^{1,p}(\X,\Y)$, the homogeneous space of $p$-Dirichlet--Sobolev maps, we will mean the collection of all maps $\Phi\in \mathrm{M}(\X,\Y)\cap \mathrm{ACC}_p(\X,\Y)$ for which there exists a $p$-weak upper gradient $G$ with $[G]_{\mX}\in \Leb_+^p(\X)$.
\end{defi}
\noindent The word ``homogeneous'' is added here, as is already mentioned in \cref{ss:Intr}, in order to highlight that there are no a priori integrability assumptions on maps in the class. The same very principle will be adopted everywhere further as well.

\begin{rem}
    Again, as with \cref{def:ACC}, the above-introduced class is sensitive to changes of maps on sets of zero measure. For more information on this, see \cite[Section 7]{HKST15}.
\end{rem}

For the sake of brevity, for any $p\in [1,+\infty)$, let us also put
\begin{equation}
    \mathrm{D}^{1,p}(\X)\coloneqq \mathrm{D}^{1,p}(\X,\R).
\end{equation}

We record the following simple property of the Dirichlet--Sobolev maps. Let $\Y=(\Y,\dY)$ be a metric space, let $p\in [1,+\infty)$, let $\Phi\colon \X\to \Y$ be a map, let $\Z=(\Z,\dZ)$ be a metric space, let $F\in \LIP(\Y,\Z)$. Then the following implication holds:
\begin{equation}\label{eq:DirLipCl}
    \Phi\in \mathrm{D}^{1,p}(\X,\Y)\qquad \Rightarrow \qquad F\circ \Phi\in \mathrm{D}^{1,p}(\X,\Z).
\end{equation}

We now proceed with the definition of minimal weak upper gradient, which can be viewed as the optimal weak upper gradient. Again, we refer the reader to \cite[Section 6.2]{HKST15} for necessary details.
\begin{defi}
Let $\Y=(\Y,\dY)$ be a metric space, let $p\in [1,+\infty)$, let $\Phi\colon \X\to \Y$ be a map that admits an $\mX$-representative $\Phi'\in \mathrm{D}^{1,p}(\X,\Y)$. By the minimal $p$-weak upper gradient of $\Phi$, we will mean the element $|\mathbf{D} \Phi|_p\in \Leb^p_+(\X)$ given by
\begin{equation}
    |\mathbf{D} \Phi|_p\coloneqq \bigwedge\limits_{G} [G]_{\mX},
\end{equation}
where the meet is taken over all $p$-weak upper gradients $G$ of $\Phi'$.
\end{defi}
\noindent The given definition can be checked to be well-posed, i.e., the element $|\mathbf{D} \Phi|_p$ does not depend on the choice of $\Phi'$, as is shown in \cite[Proposition 6.3.22]{HKST15}.

We point out the following consequence of the definition above and of properties of the meet: in the terms used in the definition, given an $\mX$-representative $G\in \mathrm{B}_+(\X)$ of $|\mathbf{D} \Phi|_p$, it holds, for $\Mod_p^{\mX}$-a.e. $\gamma\in \Cur_{\mathrm{r}}(\X)$ with $\Phi'\circ \gamma_{\mathrm{n}}\in \Cur_{\mathrm{ac}}(\Y)$, that
\begin{equation}\label{eq:MinUpGrUpGr}
    \ms_{\Phi'\circ \gamma_{\mathrm{n}}}(s)\leq G\big(\gamma_{\mathrm{n}}(s)\big)\qquad \text{for $\mathcal{L}^1$-a.e. $s\in \dom(\gamma_{\mathrm{n}})$}.
\end{equation}
In other words, any Borel representative of the minimal weak upper gradient is itself a weak upper gradient.

We will need the following simple fact about minimal weak upper gradients of maps to finite-dimensional vector spaces.
\begin{prop}\label{prop:FinDimDir}
    Let $p\in [1,+\infty)$, $M\in \Nat$, $\psi^1,...,\psi^M\in \mathrm{D}^{1,p}(\X)$. Then the following assertions are valid:
    \begin{enumerate}
    \item[$\rm I)$] letting $\Psi_{\infty}\colon \X\to \R^M_{\infty}$ be the map given by the components $\psi^1,...,\psi^M$, it holds that $\Psi_{\infty}\in \mathrm{D}^{1,p}(\X,\R^M_{\infty})$ with
        \begin{equation}
            \big|\mathbf{D} \Psi_{\infty}\big|_p(x)=\sup\limits_{m\in \overline{1,M}} \big|\mathbf{D} \psi^m\big|_p(x)\qquad \text{for $\mX$-a.e. $x\in \X$};
            \end{equation}
        \item[$\rm II)$] letting $\Psi_1\colon \X\to \R^M_1$ be the map given by the components $\psi^1,...,\psi^M$, it holds that $\Psi_1\in \mathrm{D}^{1,p}(\X,\R^M_1)$ with
        \begin{equation}
            \big|\mathbf{D} \Psi_1\big|_p(x)\leq \sum\limits_{m\in \overline{1,M}} \big|\mathbf{D} \psi^m\big|_p(x)\qquad \text{for $\mX$-a.e. $x\in \X$};
        \end{equation}
            \end{enumerate}
            \begin{proof}
               This statement follows routinely from the relevant definitions and natural properties of the metric speed, so we omit the details.
            \end{proof}
\end{prop}

Now we want to introduce a couple of measures capturing the behavior of a given map. Let $\Y=(\Y,\dY)$ be a metric space, let $p\in [1,+\infty)$, let $\Phi\in \mathrm{M}(\X,\Y)$ be a map that admits an $\mX$-representative $\Phi'\in \mathrm{D}^{1,p}(\X,\Y)$. By the $p$-energy of $\Phi$, we will mean the set function $\En_p[\Phi]\colon \mathcal{P}(\X)\to \R_+$ given for any $S\subseteq \X$ by
\begin{equation}
    \En_p[\Phi](S)\coloneqq \inf\limits_E\int\limits_E \big(|\mathbf{D} \Phi|_p\big)^p\D \mX,
\end{equation}
where the infimum is taken over all $\mX$-measurable sets $E\subseteq \X$ with $S\subseteq E$. Clearly, one then has
\begin{equation}\label{eq:EnDefEq}
    \En_p[\Phi](E)=\int\limits_E \big(|\mathbf{D} \Phi|_p\big)^p\D \mX\qquad \text{for any $\mX$-measurable set $E\subseteq \X$};
\end{equation}
based on this, one can show that the set function $\En_p[\Phi]$ forms a finite Borel-regular measure on $\X$. Next, we consider the set function $\Lambda_{\Phi,p}\colon \mathcal{P}(\Y)\to \R_+$ given for any $S\subseteq \Y$ by
\begin{equation}\label{eq:CoEnDef}
    \Lambda_{\Phi,p}(S)\coloneqq \inf\limits_{Q}\En_p[\Phi]\big(\Phi^{-1}(Q)\big),
\end{equation}
where the infimum is taken over all Borel sets $Q\subseteq \Y$ with $S\subseteq Q$. One then obviously has
\begin{equation}\label{eq:CoEnDefEq}
    \Lambda_{\Phi,p}(Q)=\En_p[\Phi]\big(\Phi^{-1}(Q)\big) \qquad \text{for any Borel set $Q\subseteq \Y$};
\end{equation}
again, this can be used to show that the set function $\Lambda_{\Phi,p}$ forms a finite Borel-regular measure on $\Y$. Thus, the (Borel part of the) measure $\Lambda_{\Phi,p}$ is nothing but the (Borel part of the) push-forward of the measure $\En_p[\Phi]$ under $\Phi$ (in the sense used, for instance, in \cite[Section 3.3]{HKST15}), the fact of which will be exploited later. Finally, we point out that both the measures $\En_p[\Phi]$ and $\Lambda_{\Phi,p}$ are stable under replacements of $\Phi$ with any of its $\mX$-representatives, which will be implicitly used later.

We are now ready to introduce the ``genuine'' class of Sobolev functions. More precisely, as the class of Dirichlet--Sobolev functions is not stable under replacements of representatives, we just create by hand a class that inherently possesses this feature.
\begin{defi}\label{def:SobFunc}
Let $p\in [1,+\infty)$. By $\mathrm{S}^{1,p}(\X)$, the homogeneous space of $p$-Sobolev functions, we will mean the collection of all functions $\phi\colon \X\to \R$ such that
\begin{equation}
    [\phi]_{\mX}\cap \mathrm{D}^{1,p}(\X)\neq \varnothing.
\end{equation}
\end{defi}
\noindent Since Dirichlet--Sobolev functions are $\mX$-measurable by definition and since any $\mX$-representative of an $\mX$-measurable function is clearly $\mX$-measurable, we deduce that Sobolev functions are themselves $\mX$-measurable as well.

\begin{rem}
   We draw the reader's attention to the absence of any identification up to a.e. equality in the above-given definition, i.e., the elements of the corresponding space are still genuine functions rather than equivalence classes.
\end{rem}

On the basis of the idea illustrated in \cref{ss:Intr}, we now introduce the classes of weakly Sobolev and Sobolev maps.
\begin{defi}\label{def:wSobSob}
Let $\Y=(\Y,\dY)$ be a metric space, let $p\in [1,+\infty)$. By $\mathrm{wS}^{1,p}(\X,\Y)$, the homogeneous space of weakly $p$-Sobolev maps, we will mean the collection of all maps $\Phi\colon \X\to \Y$ such that
\begin{equation}
    f\circ \Phi\in  \mathrm{S}^{1,p}(\X)\qquad \text{for any $f\in \LIP(\Y)$}.
\end{equation}
By $\mathrm{S}^{1,p}(\X,\Y)$, the homogeneous space of $p$-Sobolev maps, we will mean the collection of all maps $\Phi\in \mathrm{wS}^{1,p}(\X,\Y)$ such that
\begin{equation}
    \bigvee\limits_{f\in \LIP_1(\Y)} \big|\mathbf{D} (f\circ \Phi)\big|_p \in \Leb_+^p(\X).
\end{equation}
\end{defi}
\noindent Exactly as with \cref{def:SobFunc}, we first observe that the just-introduced classes are stable under replacements of representatives and then point out that there is no identification up to a.e. equality here.

From the previous definitions, we can easily deduce the following: given $p\in [1,+\infty)$, we have
\begin{equation}\label{eq:EqualScSob}
    \mathrm{S}^{1,p}(\X)=\mathrm{wS}^{1,p}(\X,\R)=\mathrm{S}^{1,p}(\X,\R).
\end{equation}
This demonstrates that the slight notation inconsistency that emerged earlier, namely when we introduced the spaces of Sobolev functions and of Sobolev maps independently of each other, is not harmful at all.

Similarly to the property given in \eqref{eq:DirLipCl}, we record the following feature of the introduced classes. Let $\Y=(\Y,\dY)$ be a metric space, let $p\in [1,+\infty)$, let $\Phi\colon \X\to \Y$ be a map, let $\Z=(\Z,\dZ)$ be a metric space, let $F\in \LIP(\Y,\Z)$. Then the following implications hold:
\begin{gather}
    \Phi\in \mathrm{wS}^{1,p}(\X,\Y)\qquad \Rightarrow \qquad F\circ \Phi\in \mathrm{wS}^{1,p}(\X,\Z);\label{eq:wSobLipSt}\\
     \Phi\in \mathrm{S}^{1,p}(\X,\Y)\qquad \Rightarrow \qquad F\circ \Phi\in \mathrm{S}^{1,p}(\X,\Z).\label{eq:SobLipSt}
\end{gather}

As an extension of the property given in \eqref{eq:EqualScSob}, it is possible to verify, via \cref{prop:FinDimDir}, the following simple fact: given $J\in \Nat_0$ and $p\in [1,+\infty)$, one has
\begin{equation}\label{eq:FinDimBanwSobSob}
    \mathrm{wS}^{1,p}(\X,\R^J_{\infty})=\mathrm{S}^{1,p}(\X,\R^J_{\infty}).
\end{equation}

Let us also mention separately the following obvious attribute: given $p\in [1,+\infty)$, $\phi\in \mathrm{S}^{1,p}(\X)$, and $c\in \R$, one has
\begin{equation}
    \phi+c\1_\X\in \mathrm{S}^{1,p}(\X)
\end{equation}
with
\begin{equation}\label{eq:InvConstAdd}
    \big|\mathbf{D}(\phi+c\1_\X)\big|_p(x)=|\mathbf{D} \phi|_p(x)\qquad \text{for $\mX$-a.e. $x\in \X$}.
\end{equation}

We now define a geometric characteristic of metric-valued maps that will play the central role in our proof strategy. In a nutshell, it detects the maximal dimension of pieces of a map visible by finite-dimensional vector-valued Lipschitz functions.
\begin{defi}\label{defi:ObsDimMap}
Let $\Y=(\Y,\dY)$ be a metric space, let $p\in [1,+\infty)$, let $\Phi\in \mathrm{wS}^{1,p}(\X,\Y)$. By the $p$-observable Hausdorff dimension of $\Phi$ and the upper $p$-observable Hausdorff dimension of $\Phi$, we will mean, respectively, the quantity
\begin{equation}
    \dim^p_{\mathrm{oH}}(\Phi)\coloneqq \sup\limits_{J,F}\inf\limits_{Q} \dim_{\mathrm{H}}(Q)
\end{equation}
and the quantity
\begin{equation}
    \overline{\dim}^p_{\mathrm{oH}}(\Phi)\coloneqq \sup\limits_{J,F}\inf\limits_{Q} \overline{\dim}_{\mathrm{H}}(Q),
\end{equation}
where the suprema are taken over all $J\in \Nat_0$ and all $F\in \LIP(\Y,\R^J_{\infty})$, while the infima are taken over all Borel sets $Q\subseteq \R^J_{\infty}$ with
\begin{equation}
\Lambda_{F\circ \Phi,p}\big(\R^J_{\infty}\backslash Q\big)=0.
\end{equation}
\end{defi}
\noindent Notice that this definition guarantees that the upper observable Hausdorff dimension, unlike the ordinary observable Hausdorff dimension, of a given map, whenever finite, is always a nonnegative integer, which is exactly the reason why we introduce these similar concepts simultaneously. In addition, analogously to the Hausdorff dimensions of sets, it holds, in terms of the data above, that
\begin{equation}
    \dim_{\mathrm{oH}}^p(\Phi)\leq \overline{\dim}_{\mathrm{oH}}^p(\Phi)\leq \dim_{\mathrm{oH}}^p(\Phi)+1,
\end{equation}
the fact of which will be implicitly used later.

Below we recall the notion of Luzin--Lipschitz property of a map. We refer the reader to \cite{GT21} and \cite{O25} for instances of the use of this concept.
\begin{defi}
Let $\Y=(\Y,\dY)$ be a metric space. A map $\Phi\colon \X\to \Y$ will be said to have the Luzin--Lipschitz property if there are compact sets $K_i\subseteq \X$, $i\in \Nat$, with
\begin{equation}
    \mX\bigg(\X\Big\backslash \bigcup\limits_{i\in \Nat}K_i\bigg)=0
\end{equation}
such that
\begin{equation}
    \text{the map $\Phi|_{K_i}$ is Lipschitz for each $i\in \Nat$}.
\end{equation}
\end{defi}

\begin{rem}
    The definition above can vary in certain aspects sometimes, like one can require the decomposition to consist of merely Borel sets rather than compact sets. Under our choice of what we mean by a metric measure space, there is no difference, as can be easily verified. In general, though, on ``incomplete metric measure spaces'', the two approaches are not exactly the same.
\end{rem}

In many geometrically interesting situations, any real-valued Sobolev function turns out to have the Luzin--Lipschitz property. It therefore makes sense to abstract this attribute into a separate characteristic of the underlying space, which is exactly what we do below. The corresponding name of the property is chosen by analogy with the well-known Sobolev-to-Lipschitz property appearing in the theory of $\mathrm{RCD}$-spaces; see \cite[Definition 6.1.2]{GP20} for more on this.
\begin{defi}\label{def:SobLL}
Let $p\in [1,+\infty)$. The space $\X$ will be said to satisfy the $p$-Sobolev-to-Luzin--Lipschitz property if any function in $\mathrm{S}^{1,p}(\X)$ has the Luzin--Lipschitz property.
\end{defi}

\begin{rem}
    As it is mentioned briefly in \cref{ss:Intr}, a central class of metric measure spaces satisfying the Sobolev-to-Luzin--Lipschitz property is represented by those that possess certain variants of the doubling and Poincar{\'e} inequalities. In fact, the local version of both would suffice, as shown in \cite[Lemma 8.5]{BB18}. In particular, any domain in a Riemannian manifold, as well as in any sub-Riemannian manifold, under the natural metric and measure, falls into this admissible class. On top of that, we note also that even some reasonable ``infinite-dimensional'' metric measure spaces turn out to enjoy such a feature; see \cite{ABT18} for the details.
\end{rem}

Let us mention separately that the just-defined property passes to vector-valued functions as well, which will be implicitly used later. Namely, given $p\in [1,+\infty)$, if the space $\X$ satisfies the $p$-Sobolev-to-Luzin--Lipschitz property, then, given $J\in \Nat_0$ and $\Psi\in \mathrm{S}^{1,p}(\X,\R^J_{\infty})$, the function $\Psi$ has the Luzin--Lipschitz property. This is an immediate consequence of the relevant definitions.

\section{Auxiliary statements}
Here we list several known facts that will form the basis for our subsequent reasoning.

We start with the following folklore feature of the lattice of measurable functions.
\begin{prop}\label{prop:EssSupInt}
    Let $\X=(\X,\dX,\mX)$ be a metric measure space, let $\mathscr{V}\subseteq \Leb_+(\X)$ be a nonempty set, let $\eta\in \R_+$ be such that
    \begin{equation}
        \sup\limits_{\mathcal{V}}\int\limits_{\X} \Bigg(\bigvee\limits_{v\in \mathcal{V}} v \Bigg)\D \mX\leq \eta,
    \end{equation}
    where the supremum is taken over all nonempty finite families $\mathcal{V}\subseteq \mathscr{V}$. Then one has
    \begin{equation}
        \int\limits_{\X}\Bigg(\bigvee\limits_{v\in \mathscr{V}} v\Bigg)\D\mX \leq \eta.
    \end{equation}
    In particular, one has
    \begin{equation}
        \bigvee\limits_{v\in \mathscr{V}} v\in \Leb^1_+(\X).
    \end{equation}
    \begin{proof}
        The proof can be found, for instance, in \cite[Lemma B.1]{SV25}.
    \end{proof}
\end{prop}
\noindent Clearly, the provided attribute has nothing to do with exactly the metric measure space structure, as all the same is valid on any measure space.

Below we record a lemma that, while being mostly technical in its essence, is truly crucial for the proof methods that we will use later. There, a finite Borel measure is constructed on the target space of a weakly Sobolev map that controls the energy of the real-valued Lipschitz post-compositions of the map on the preimages of arbitrary Borel sets, importantly without an a priori assumption that the map is Sobolev. The corresponding statement should also be compared with its ``classical'' precursor in the form depicted, for instance, in \cite{VU02}, where quite an analogous object was constructed, moreover in application to a problem quite similar in spirit to the one under discussion in our paper.
\begin{lem}\label{lem:ExistsCoMeas}
    Let $\X=(\X,\dX,\mX)$ be a metric measure space, let $\Y=(\Y,\dY)$ be a metric space, let $p\in [1,+\infty)$, let $\Phi\in \mathrm{M}(\X,\Y)\cap \mathrm{wS}^{1,p}(\X,\Y)$ be such that
          \begin{equation}\label{eq:ExistsCoMeas0}
            \sup\limits_{f\in \LIP_1(\Y)} \En_p[f\circ \Phi](\X)<+\infty.
        \end{equation}
    Then the set function $\Theta_{\Phi,p}\colon \mathcal{P}(\Y)\to \R_+$ given for any $S\subseteq \Y$ by
    \begin{equation}
        \Theta_{\Phi,p}(S)=\inf\limits_{Q}\sup\limits_{f\in \LIP_1(\Y)} \En_p[f\circ \Phi]\big(\Phi^{-1}(Q)\big),
    \end{equation}
    where the infimum is taken over all Borel sets $Q\subseteq \Y$ with $S\subseteq Q$, is a finite Borel-regular measure on $\Y$.
    \begin{proof}
        The provided statement can be verified in essentially the same way as in \cite[Theorem 3.5]{SV25}.
    \end{proof}
\end{lem}
\noindent As will be shown later, namely in \cref{lem:EquiBEn}, the additional condition in \eqref{eq:ExistsCoMeas0} can be remarkably removed, i.e., it follows automatically from all the other ones.

Let us point out separately the following straightforward properties of the above-defined objects that will be of use later. Let $\X,\Y,p,\Phi$ be as in the lemma. First, we have
\begin{equation}\label{eq:PropCoMeas1}
    \Theta_{\Phi,p}(Q)=\sup\limits_{f\in \LIP_1(\Y)} \En_p[f\circ \Phi]\big(\Phi^{-1}(Q)\big)\qquad \text{for any Borel set $Q\subseteq \Y$}.
\end{equation}
Then, given $\Phi'\in[\Phi]_{\mX}$, we have
\begin{equation}\label{eq:PropCoMeas2}
    \Theta_{\Phi,p}(Q)=\Theta_{\Phi',p}(Q)\qquad \text{for any Borel set $Q\subseteq \Y$}.
\end{equation}
Finally, given a metric space $\Z=(\Z,\dZ)$ and $F\in \LIP_1(\Y,\Z)$, we have
\begin{equation}\label{eq:PropCoMeas3}
    \Theta_{F\circ \Phi,p}(Q)\leq \Theta_{\Phi,p}\big(F^{-1}(Q)\big) \qquad \text{for any Borel set $Q\subseteq \Z$}.
\end{equation}

Now we provide the main tool for our forthcoming analysis. The corresponding statement says basically that a metric space, under an appropriate dimensional bound on it, in the presence of any ``reasonable measure'' can be partitioned into a countable collection of certain ``nondegenerate'' charts, which, in turn, are expressed in terms of specific vector-valued Lipschitz functions on the space that ``effectively'' control the speed of almost every curve therein. As was indicated already in \cref{ss:Intr}, the condition below is quite instructive to compare with the BLD condition in the sense of \cite{LD13} due to significant ideological similarities between the two.
\begin{theo}\label{theo:BLD}
   Let $\Y=(\Y,\dY)$ be a separable, locally complete metric space, let $d\in \Nat_0$, let $Q\subseteq \Y$ be a Borel set with $\dim_{\mathrm{H}}(Q)\leq d$, let $\nu$ be a locally finite Borel measure on $\Y$, let $\epsilon\in (0,1)$. Then there exist disjoint Borel sets $Q_i\subseteq Q$, $i\in \Nat$, and functions $H^i_{\infty}\in \LIP_1(\Y,\R^d_{\infty})$, $i\in \Nat$, such that the following conditions hold:
   \begin{itemize}
       \item[$\rm I)$] one has
       \begin{equation}
           \nu\bigg(Q\Big\backslash\bigcup\limits_{i\in \Nat}Q_i\bigg)=0;
       \end{equation}
       \item[$\rm II)$] for each $i\in \Nat$, letting $H^i_1\in \LIP(\Y,\R^d_1)$ be the map acting as $H^i_{\infty}$, one has, for $\Mod^{\nu}_{\infty}$-a.e. $\beta\in \Cur_{\mathrm{L}}(\Y)$, that
       \begin{equation}\label{eq:BLD0}
           \ms_{\beta}(t)\leq (1+\epsilon)\ms_{H^i_1\circ \beta}(t)\qquad \text{for $\mathcal{L}^1$-a.e. $t\in \beta^{-1}(Q_i)$}.
       \end{equation}
   \end{itemize}
   \begin{proof}
       The presented statement, up to certain adjustments, is in fact one of the main results of the previous work by the author of the current manuscript, namely \cite{O26.2}. To derive exactly the desired formulation, one just needs to combine assertion I) of \cite[Theorem 5.30]{O26.2} with \cite[Definition 5.18]{O26.2} and \cite[Theorem 5.13]{O26.2}.
   \end{proof}
\end{theo}

\section{Energy bounds}

Here we establish a remarkable self-improvement effect for energies of weakly Sobolev maps.

In this section, a metric measure space $\X=(\X,\dX,\mX)$, a nonempty metric space $\Y=(\Y,\dY)$, and an exponent $p\in [1,+\infty)$ are assumed fixed.

For what follows, let us fix $\widebar{y}\in \Y$. Let then $\mathbb{L}$ be the vector subspace of $\LIP(\Y)$ made of all $f\in \LIP(\Y)$ with $f(\widebar{y})=0$. We also consider on $\mathbb{L}$ the norm $\|\cdot\|_{\mathbb{L}}$ given for any $f\in \mathbb{L}$ by
\begin{equation}
    \|f\|_{\mathbb{L}}\coloneqq \inf\Big\{L\in \R_+\bigm| \text{$\big|f(y')-f(y)\big|\leq L\dY(y,y')$ for any $y,y'\in \Y$}\Big\};
\end{equation}
the fact that this assignment indeed defines a norm on $\mathbb{L}$ is trivial to check. We thus get at our disposal the normed vector space $\mathbb{L}=\big(\mathbb{L},\|\cdot\|_{\mathbb{L}}\big)$. We send the reader to \cite[Section 2]{W18} for the necessary details.

\begin{prop}\label{prop:SpLipFunc}
    The space $\mathbb{L}$ is a Banach space. In addition, given $f\in \mathbb{L}$ and $(f_n)_{n\in \Nat}\subseteq \mathbb{L}$ such that the sequence $(f_n)_{n\in \Nat}$ converges in $\mathbb{L}$ to $f$, the sequence $(f_n)_{n\in \Nat}$ converges pointwise on $\Y$ to $f$.
    \begin{proof}
        The first part is verified in \cite[Proposition 2.3]{W18}. The second part can be easily checked directly.
    \end{proof}
\end{prop}

It is well known that the space $\Leb(\X)$ equipped with the topology of the local convergence in $\mX$-measure on $\X$ is an $\mathrm{F}$-space, i.e., it is a topological vector space whose topology is induced by a translation-invariant complete metric on $\Leb(\X)$; this fact can be found in \cite[Section 1.1.2]{GP20}. So, let $\mathsf{d}_{\Leb(\X)}$ be any such metric.

The space $\mathrm{S}^{1,p}(\X)$ clearly forms a vector space. We consider then the seminorm $\|\cdot\|_{\mathrm{S}^{1,p}(\X)}$ on $\mathrm{S}^{1,p}(\X)$ given for any $\phi\in \mathrm{S}^{1,p}(\X)$ by
\begin{equation}
    \|\phi\|_{\mathrm{S}^{1,p}(\X)}\coloneqq \big(\En_p[\phi](\X)\big)^{\frac{1}{p}};
\end{equation}
the fact that this assignment indeed defines a seminorm on $\mathrm{S}^{1,p}(\X)$ is standard.

Let $\mathbb{V}$ be the vector subspace of $\Leb(\X)$ made of all $v\in \Leb(\X)$ for which there is $\phi\in \mathrm{S}^{1,p}(\X)$ with $v=[\phi]_{\mX}$. In more transparent terms, the space $\mathbb{V}$ consists of all the elements of $\mathrm{S}^{1,p}(\X)$ but identified up to $\mX$-a.e. equality. We consider also the metric $\mathsf{d}_{\mathbb{V}}$ on $\mathbb{V}$ given for all $\phi_1,\phi_2\in \mathrm{S}^{1,p}(\X)$ by
\begin{equation}
    \mathsf{d}_{\mathbb{V}}\big([\phi_1]_{\mX},[\phi_2]_{\mX}\big)\coloneqq \mathsf{d}_{\Leb(\X)}\big([\phi_1]_{\mX},[\phi_2]_{\mX}\big)+\|\phi_2-\phi_1\|_{\mathrm{S}^{1,p}(\X)}.
\end{equation}
The fact that this assignment is well posed, i.e., the quantity on the left does not depend on the chosen representatives, and indeed defines a metric on $\mathbb{V}$ can be easily verified. We now endow $\mathbb{V}$ with the topology induced by $\mathsf{d}_{\mathbb{V}}$. It is then clear that this turns $\mathbb{V}$ into a pre-$\mathrm{F}$-space, i.e., it is a topological vector space whose topology is induced by a translation-invariant metric on $\mathbb{V}$. 

The statement below is the cornerstone one for this section. While it may seem standard and is perhaps well known among the experts, it is not apparently presented in this exact form in the literature.
\begin{lem}\label{lem:HomSobSpCompl}
    The space $\mathbb{V}$ is an $\mathrm{F}$-space.
    \begin{proof}
        It definitely suffices to verify that the metric $\mathsf{d}_{\mathbb{V}}$ is complete. The core of the reasoning lies in the following important circumstance: the approach to the Sobolev calculus adopted in the current manuscript, the Newtonian one, is fully compatible with the so-called test-plan approach. Namely, the two approaches yield entirely the same homogeneous Sobolev spaces (as sets) and the same minimal weak upper gradients. In the greatest generality, albeit under slightly stronger requirements for the underlying metric measure space that are immaterial for our purposes, this pivotal result is established in \cite[Theorem 7.1]{AILP24}, at least for functions with a priori integrability assumptions. This constraint, in turn, can be easily removed by using the standard strong locality property of the minimal weak upper gradients, which is given by \cite[Proposition 6.3.22]{HKST15} for the Newtonian approach and by \cite[Theorem 2.1.28]{GP20} for the test-plan approach. Once this fact is taken into account, what remains is to check the stated property in terms of the test-plan definition. This, in turn, can be extracted from the proofs given in \cite[Proposition 2.1.13]{GP20} and \cite[Theorem 2.1.17]{GP20}. We omit further details for the sake of brevity of the exposition.
    \end{proof}
\end{lem}

It is convenient to introduce the following notation. Given $\Phi\in \mathrm{wS}^{1,p}(\X,\Y)$, by $\mathfrak{C}_{\Phi,p}$, we will denote the linear operator $\mathbb{L}\to \mathbb{V}$ sending $f\in \mathbb{L}$ to $[f\circ \Phi]_{\mX}$.

With the adopted topologies on $\mathbb{L}$ and $\mathbb{V}$, we get the following key property of the composition operator above.
\begin{prop}\label{prop:CompOpCl}
    The operator $\mathfrak{C}_{\Phi,p}$ is closed.
    \begin{proof}
        Let $f\in \mathbb{L}$, $\phi\in \mathrm{S}^{1,p}(\X)$, $(f_n)_{n\in \Nat}\subseteq \mathbb{L}$ be such that the sequence $(f_n)_{n\in \Nat}$ converges to $f$ in $\mathbb{L}$ and the sequence $\big([f_n\circ \Phi]_{\mX}\big)_{n\in \Nat}$ converges to $[\phi]_{\mX}$ in $\mathbb{V}$. By \cref{prop:SpLipFunc}, it then follows that the sequence $(f_n)_{n\in \Nat}$ converges to $f$ pointwise on $\Y$, whence the sequence $(f_n\circ \Phi)_{n\in \Nat}$ converges pointwise on $\X$ to $f\circ \Phi$. On the other hand, we observe that the convergence of $\big([f_n\circ \Phi]_{\mX}\big)_{n\in \Nat} $ in $\mathbb{V}$ to $[\phi]_{\mX}$ implies that the sequence $(f_n\circ \Phi)_{n\in \Nat}$ converges locally in $\mX$-measure on $\X$ to $\phi$, whence, by the Riesz theorem, for some increasing sequence $(n_k)_{k\in \Nat}\subseteq \Nat$, the sequence $(f_{n_k}\circ \Phi)_{k\in \Nat}$ converges $\mX$-a.e. on $\X$ to $\phi$. These two facts clearly imply that $[f\circ \Phi]_{\mX}=[\phi]_{\mX}$. By definition, this implies the closedness of the operator, so the proof is complete.
    \end{proof}
\end{prop}

The lemma below provides the main outcome of the current section. This statement should be compared with \cite[Lemma 3.2]{OT24}, where a similar phenomenon, in a much simpler setting though, was achieved by slightly more direct methods, as well as with \cite[Proposition 27]{P26}, where essentially the same attribute is verified, yet in the classical framework only.
\begin{lem}\label{lem:EquiBEn}
    Let $\Phi\in \mathrm{wS}^{1,p}(\X,\Y)$. Then one has
    \begin{equation}\label{eq:EquiBEn0}
        \sup\limits_{f\in \LIP_1(\Y)} \En_p[f\circ \Phi](\X)<+\infty.
    \end{equation}
    \begin{proof}
        Since the spaces $\mathbb{L}$ and $\mathbb{V}$ are $\mathrm{F}$-spaces by \cref{prop:SpLipFunc} and \cref{lem:HomSobSpCompl}, the operator $\mathfrak{C}_{\Phi,p}$ is therefore a closed linear operator between $\mathrm{F}$-spaces by \cref{prop:CompOpCl}, so, by the closed graph theorem in the form given, for instance, in \cite[Theorem 14.3.4]{NB10}, the operator $\mathfrak{C}_{\Phi,p}$ is also continuous and hence bounded, whence, rewriting everything properly, we can easily get that
        \begin{equation}
            \sup\limits_{\substack{f\in \mathbb{L}\\\|f\|_{\mathbb{L}}\leq 1}} \En_p[f\circ \Phi](\X)<+\infty.
        \end{equation}
       Since it obviously holds, for any $f\in \LIP_1(\Y)$, that
       \begin{equation}
           \big(f-f(\widebar{y})\1_{\Y}\big)\in \mathbb{L}
       \end{equation}
       and that
       \begin{equation}
           \big\|f-f(\widebar{y})\1_{\Y}\big\|_{\mathbb{L}}\leq 1,
       \end{equation}
       we immediately obtain, from the property given in \eqref{eq:InvConstAdd}, the estimate in \eqref{eq:EquiBEn0}. Thus the proof is fully complete.
    \end{proof}
\end{lem}

\section{Key estimates}
Here we obtain some cornerstone estimates between the minimal upper gradients of a map and of its post-composition with a Lipschitz map based on application of \cref{theo:BLD}.

In this section, a metric measure space $\X=(\X,\dX,\mX)$ and an exponent $p\in [1,+\infty)$ are assumed fixed.

The first step is to derive the following result about the transformation rule for moduli under a Dirichlet--Sobolev map. It might be well known among the experts in the field, but we prefer to give all the details for the sake of completeness of our exposition.
\begin{lem}\label{lem:ModCurvTrans}
    Let $\Y=(\Y,\dY)$ be a metric space, let $\Phi\in \mathrm{D}^{1,p}(\X,\Y)$, let $\mathit{B}\subseteq \Cur_{\mathrm{r}}(\Y)$ be such that $\Mod^{\Lambda_{\Phi,p}}_{\infty}(\mathit{B})=0$. Then it holds, for $\Mod_p^{\mX}$-a.e. $\gamma\in \Cur_{\mathrm{r}}(\X)$, that
    \begin{equation}
        \Phi\circ \gamma_{\mathrm{n}} \not\in \mathit{B}.
    \end{equation}
    \begin{proof}
        We consider the family
        \begin{equation}
            \G\coloneqq \Big\{\gamma\in \Cur_{\mathrm{r}}(\X)\bigm| \Phi\circ \gamma_{\mathrm{n}}\in \mathit{B}\Big\}.
        \end{equation}
       Our corresponding goal is to check that $\Mod_p^{\mX}(\G)=0$.

        Fix $\varepsilon \in (0,1)$. As $\Mod_{\infty}^{\Lambda_{\Phi,p}}(\mathit{B})=0$, by definition, there is a function $\rho_{\Y}\in \overline{\mathrm{B}}_+(\Y)$ admissible for $\mathit{B}$ such that
        \begin{equation}
            \esup\limits_{\Lambda_{\Phi,p}}(\rho_{\Y})\leq \varepsilon.
        \end{equation}
        We then find $\mX$-representatives $G,\rho_{\Phi}\in \mathrm{B}_+(\X)$ of, respectively, $|\mathbf{D} \Phi|_p$ and $\rho_{\Y}\circ \Phi$, which is possible by the property given in \eqref{eq:BorRepr}. After that, we consider the Borel function $\rho_{\X}\in \mathrm{B}_+(\X)$ given by
        \begin{equation}
            \rho_{\X}\coloneqq \rho_{\Phi} G.
        \end{equation}
        Let also $E\subseteq \X$ be an $\mX$-measurable set with $\mX(E)=0$ such that
        \begin{equation}
            \rho_{\Phi}(x)=\rho_{\Y}\big(\Phi(x)\big)\qquad \text{for any $x\in \big(\X\backslash E\big)$},
        \end{equation}
        and consider then the family $\G_E$ of all $\gamma\in \G$ for which the equality
        \begin{equation}
            \mathcal{L}^1\big(\gamma_{\mathrm{n}}^{-1}(E)\big)=0,
        \end{equation}
        the inclusion
        \begin{equation}
            \Phi\circ \gamma_{\mathrm{n}}\in \Cur_{\mathrm{ac}}(\Y),
        \end{equation}
        and the condition
        \begin{equation}
            \ms_{\Phi\circ \gamma_{\mathrm{n}}}(s)\leq G\big(\gamma_{\mathrm{n}}(s)\big)\qquad \text{for $\mathcal{L}^1$-a.e. $s\in \dom(\gamma_{\mathrm{n}})$}
        \end{equation}
        hold. With the help of \cref{prop:ModZeroNegSet}, we can verify that
        \begin{equation}\label{eq:ModCurvTrans1}
        \Mod_p^{\mX}\big(\G\backslash \G_E\big)=0.
        \end{equation}
        We claim now that the function $\rho_{\X}$ is admissible for $\G_E$. Indeed, due to the property given in \eqref{eq:ParInvLinInt}, the fact that the function $\rho_{\Y}$ is admissible for $\mathit{B}$, and the defining properties of $\G_E$, we can write, for any $\gamma\in \G_E$, that
        \begin{equation}
        \begin{gathered}
            \int\limits_{\gamma}\rho_{\X}=\int\limits_{\dom(\gamma_{\mathrm{n}})} \rho_{\Phi}\big(\gamma_{\mathrm{n}}(s)\big)G\big(\gamma_{\mathrm{n}}(s)\big)\D \mathcal{L}^1(s)=\\
            =\int\limits_{\dom(\gamma_{\mathrm{n}})} \rho_{\Y}\Big(\Phi\big(\gamma_{\mathrm{n}}(s)\big)\Big) G\big(\gamma_{\mathrm{n}}(s)\big)\D \mathcal{L}^1(s)\geq \\
            \geq \int\limits_{\dom(\gamma_{\mathrm{n}})} \rho_{\Y}\Big(\Phi\big(\gamma_{\mathrm{n}}(s)\big)\Big) \ms_{\Phi\circ \gamma_{\mathrm{n}}}(s)\D \mathcal{L}^1(s)=\int\limits_{\Phi\circ \gamma_{\mathrm{n}}} \rho_{\Y}\geq 1.
            \end{gathered}
        \end{equation}
In turn, we can write, via the change-of-variable formula for a measure and its push-forward under a map and the property given in \eqref{eq:CoEnDefEq}, that
        \begin{equation}
            \int\limits_{\X} \rho_{\X}^p\D \mX=\int\limits_{\Y} \rho_{\Y}^p\D \Lambda_{\Phi,p}\leq \varepsilon^p \Lambda_{\Phi,p}(\Y)=\varepsilon^p\En_p[\Phi](\X).
        \end{equation}
        By the very definition of the $p$-modulus, we get that
        \begin{equation}
            \Mod_p^{\mX}(\G_E)\leq \varepsilon^p\En_p[\Phi](\X),
        \end{equation}
        whence, thanks to \eqref{eq:ModCurvTrans1}, we deduce that
        \begin{equation}
            \Mod_p^{\mX}(\G)\leq \varepsilon^p\En_p[\Phi](\X).
        \end{equation}
        Due to the arbitrariness of $\varepsilon$ and the inclusion $\Phi\in \mathrm{D}^{1,p}(\X,\Y)$, we can clearly conclude that
        \begin{equation}
            \Mod_p^{\mX}(\G)=0,
        \end{equation}
        as desired. This finishes the proof.
    \end{proof}
\end{lem}

The main outcome of \cref{theo:BLD} is represented by the condition given in \eqref{eq:BLD0}, which, however, is merely proven to hold for almost every Lipschitz curve, rather than for almost every rectifiable curve. In the technical proposition below, we resolve this issue.
\begin{prop}\label{prop:CurvFromLipToRect}
    Let $\Y=(\Y,\dY)$ be a metric space, let $\nu$ be a locally finite Borel measure on $\Y$, let $\Z=(\Z,\dZ)$ be a metric space, let $F\in \LIP(\Y,\Z)$, let $Q\subseteq \Y$ be a Borel set, let $w\in \mathrm{B}_+(\Y)$. Suppose it holds, for $\Mod_{\infty}^{\nu}$-a.e. $\beta\in \Cur_{\mathrm{L}}(\Y)$, that
    \begin{equation}\label{eq:CurvFromLipToRect0}
        \ms_{\beta}(t)\leq w\big(\beta(t)\big)\ms_{F\circ \beta}(t)\qquad \text{for $\mathcal{L}^1$-a.e. $t\in \beta^{-1}(Q)$}.
    \end{equation}
    Then it holds, for $\Mod_{\infty}^{\nu}$-a.e. $\beta\in \Cur_{\mathrm{r}}(\Y)$, that
    \begin{equation}\label{eq:CurvFromLipToRect00}
        \ms_{\beta}(t)\leq w\big(\beta(t)\big)\ms_{F\circ \beta}(t)\qquad \text{for $\mathcal{L}^1$-a.e. $t\in \beta^{-1}(Q)$}.
    \end{equation}
    \begin{proof}
    Let $\mathit{B}_0$ be the family of all $\beta\in \Cur_{\mathrm{L}}(\Y)$ for which the condition in \eqref{eq:CurvFromLipToRect0} does not hold, so that we have $\Mod_{\infty}^{\nu}( \mathit{B}_0)=0$ by the premise. Consider then the collection
        \begin{equation}\label{eq:CurvFromLipToRect1}
            \mathit{B}\coloneqq \Big\{\beta\in \Cur_{\mathrm{r}}(\Y)\bigm|\beta_{\mathrm{n}}\in \mathit{B}_0\Big\}.
        \end{equation}
        From the property of the $\infty$-modulus given in \cref{prop:ModEst}, it follows that
\begin{equation}\label{eq:CurvFromLipToRect2}
    \Mod_{\infty}^{\nu}(\mathit{B})\leq \Mod_{\infty}^{\nu}(\mathit{B}_0)=0.
\end{equation}
Fix now $\beta\in \big(\Cur_{\mathrm{r}}(\Y)\backslash\mathit{B}\big)$, so that we have the inclusion $\beta_{\mathrm{n}}\in \big(\Cur_{\mathrm{L}}(\Y)\backslash \mathit{B}_0\big)$. By \cref{prop:NatParam}, we know that
\begin{equation}
    \beta_{\mathrm{n}}\circ \sigma_{\beta}=\beta.
\end{equation}
Applying then \cref{prop:ChainRulCur} to $\beta_{\mathrm{n}}$ with $\sigma_{\beta}$ and to $F\circ \beta_{\mathrm{n}}$ with $\sigma_{\beta}$, we can get that
\begin{equation}
    \ms_{\beta}(t)=\ms_{\beta_{\mathrm{n}}}\big(\sigma_{\beta}(t)\big)(\partial\sigma_{\beta})(t)\qquad \text{for $\mathcal{L}^1$-a.e. $t\in \dom(\beta)$}
\end{equation}
and
\begin{equation}
    \ms_{F\circ \beta}(t)=\ms_{F\circ \beta_{\mathrm{n}}}\big(\sigma_{\beta}(t)\big)(\partial\sigma_{\beta})(t)\qquad \text{for $\mathcal{L}^1$-a.e. $t\in \dom(\beta)$}.
\end{equation}
Combining these two relations, and the same proposition once again, with the condition in \eqref{eq:CurvFromLipToRect0} for $\beta_{\mathrm{n}}$, we can achieve exactly the condition in \eqref{eq:CurvFromLipToRect00} for $\beta$. This, in light of \eqref{eq:CurvFromLipToRect2}, finishes the proof.
\end{proof}
\end{prop}

In the lemma below, we derive, for a Dirichlet--Sobolev map and its post-composition with a Lipschitz map, an estimate in terms of the minimal upper gradients from a corresponding estimate in terms of the metric speeds. In a sense, this statement contains the cornerstone idea lying behind our overall reasoning, namely that a control of curves in a space can be suitably transferred to a control of maps into that space.
\begin{lem}\label{lem:FromMetrSpToUpGr}
    Let $\Y=(\Y,\dY)$ be a metric space, let $\Phi\in \mathrm{D}^{1,p}(\X,\Y)$, let $\Z=(\Z,\dZ)$ be a metric space, let $F\in \LIP(\Y,\Z)$, let $Q\subseteq \Y$ be a Borel set, let $w\in \mathrm{B}_+(\Y)$. Suppose it holds, for $\Mod_{\infty}^{\Lambda_{\Phi,p}}$-a.e. $\beta\in \Cur_{\mathrm{L}}(\Y)$, that
    \begin{equation}\label{eq:FromMetrSpToUpGr0}
        \ms_{\beta}(t)\leq w\big(\beta(t)\big) \ms_{F\circ \beta}(t)\qquad \text{for $\mathcal{L}^1$-a.e. $t\in \beta^{-1}(Q)$}.
    \end{equation}
    Then the following condition holds:
    \begin{equation}\label{eq:FromMetrSpToUpGr00}
        |\mathbf{D} \Phi|_p(x)\leq w\big(\Phi(x)\big)\big|\mathbf{D}(F\circ \Phi)\big|_p(x)\qquad \text{for $\mX$-a.e. $x\in \Phi^{-1}(Q)$}.
    \end{equation}
    \begin{proof}
    To start with, we note that the inclusion $F\circ \Phi\in \mathrm{D}^{1,p}(\X,\Z)$ holds due to the property given by \eqref{eq:DirLipCl}. Next, we find $\mX$-representatives $W,G_{\Phi},G_{F\circ \Phi}\in \mathrm{B}_+(\X)$ of, respectively, $w\circ \Phi$, $|\mathbf{D} \Phi|_p$, and $\big|\mathbf{D}(F\circ \Phi)\big|_p$, which is possible by the property given in \eqref{eq:BorRepr}. Find also a Borel set $Q_{\X}\subseteq \X$ with $\Phi^{-1}(Q)\subseteq Q_{\X}$ such that $\mX\big(Q_{\X}\backslash \Phi^{-1}(Q)\big)=0$, which is possible due to the fact that the measure $\mX$ is Borel-regular. Consider then the function $G\in \mathrm{B}_+(\X)$ given for any $x\in \X$ by
    \begin{equation}
        G(x)\coloneqq \begin{dcases}
        W(x)G_{F\circ \Phi}(x),\qquad x\in Q_{\X},\\
            G_{\Phi}(x),\qquad \qquad\quad\;\,\text{otherwise}.
        \end{dcases}
    \end{equation}
    Our corresponding goal is to prove that this function is a $p$-weak upper gradient of $\Phi$, which clearly suffices to obtain the condition in \eqref{eq:FromMetrSpToUpGr00}.
    
    First, we find an $\mX$-measurable set $E\subseteq \X$ with $\mX(E)=0$ and $\big(Q_{\X}\backslash \Phi^{-1}(Q)\big)\subseteq E$ such that
    \begin{equation}
        W(x)=w\big(\Phi(x)\big)\qquad \text{for any $x\in \big(\X\backslash E\big)$}.
    \end{equation}
    We observe then, thanks to \cref{prop:CurvFromLipToRect}, that the condition in \eqref{eq:FromMetrSpToUpGr0} holds for $\Mod_{\infty}^{\Lambda_{\Phi,p}}$-a.e. $\beta\in \Cur_{\mathrm{r}}(\Y)$. Next, let $\G$ be the family of all $\gamma\in \Cur_{\mathrm{r}}(\X)$ with
    \begin{equation}
        \mathcal{L}^1\big(\gamma_{\mathrm{n}}^{-1}(E)\big)=0
    \end{equation}
    such that for $\beta= \Phi\circ \gamma_{\mathrm{n}}$ the inclusion $\beta\in \Cur_{\mathrm{ac}}(\Y)$, the condition in \eqref{eq:FromMetrSpToUpGr0}, the condition
    \begin{equation}
        \ms_{\beta}(s)\leq G_{\Phi}\big(\gamma_{\mathrm{n}}(s)\big)\qquad \text{for $\mathcal{L}^1$-a.e. $s\in \dom(\gamma_{\mathrm{n}})$},
    \end{equation}
    and the condition
    \begin{equation}
        \ms_{F\circ \beta}(s)\leq G_{F\circ \Phi}\big(\gamma_{\mathrm{n}}(s)\big)\qquad \text{for $\mathcal{L}^1$-a.e. $s\in \dom(\gamma_{\mathrm{n}})$}
    \end{equation}
    hold. Taking \cref{prop:ModZeroNegSet} and \cref{lem:ModCurvTrans} into account, we see that
    \begin{equation}
        \Mod_{p}^{\mX}\big(\Cur_{\mathrm{r}}(\X)\backslash \G\big)=0.
    \end{equation}
    We now fix $\gamma\in \G$ and, combining all the properties invoked above, deduce that
    \begin{equation}
        \ms_{\Phi\circ \gamma_{\mathrm{n}}}(s)\leq G(\gamma_{\mathrm{n}}(s))\qquad \text{for $\mathcal{L}^1$-a.e. $s\in \dom(\gamma_{\mathrm{n}})$}.
    \end{equation}
    This allows us to conclude that the function $G$ is indeed a $p$-weak upper gradient of $\Phi$, the fact of which, as already discussed, finishes the proof.
    \end{proof}
\end{lem}

Now we are ready to establish the theorem containing the central estimate of our whole exposition, which basically stems from the application of \cref{theo:BLD} in combination with \cref{lem:FromMetrSpToUpGr}. In its essence, this statement is of a similar flavor as the structural relation obtained in \cite[Theorem 21]{P26} (forming moreover the core of the corresponding work) that provides a universal control of the modulus of the gradient of an arbitrary vector-valued Sobolev function (on a $1$-dimensional or $2$-dimensional Euclidean domain) in terms of the modulus of the gradient of its post-composition with a specially constructed ``tracking'' vector-valued Lipschitz function whose codomain is notably of the same dimension as the source domain. Furthermore, quite an analogous relation, and in our much more general setting, can in fact be extracted, after a bit of necessary reasoning, from the forthcoming proof as well.
\begin{theo}\label{theo:EstEnCoMeas}
    Let $\Y=(\Y,\dY)$ be a separable, complete metric space, let $\Phi\in \mathrm{D}^{1,p}(\X,\Y)$. Then the assertions below hold.
    \begin{itemize}
    \item[$\rm I)$] Let $d\in \Nat_0$, let $Q\subseteq \Y$ be a Borel set with $\dim_{\mathrm{H}}(Q)\leq d$. Then the following estimate holds:
\begin{equation}\label{eq:EstEnCoMeas0}
    \Lambda_{\Phi,p}(Q)\leq d^p \Theta_{\Phi,p}(Q).
\end{equation}
    \item[$\rm II)$] Let $d\in \Nat_0$, let $Q\subseteq \Y$ be a Borel set with $\dim_{\mathrm{H}}(Q)\leq d$ such that, for some $\mX$-measurable set $E\subseteq \X$ with $\En_p[\Phi](E)=0$, one has $\Phi\big(\X\backslash E\big)\subseteq Q$. Then the following estimate holds:
    \begin{equation}\label{eq:EstEnCoMeas00}
        \En_p[\Phi](\X)\leq d^p\Theta_{\Phi,p}(\Y).
    \end{equation}
    \end{itemize}
    \begin{proof}
        We will assume further that $d>0$ not to make separate technical amendments to our reasoning when $d=0$, which is in fact substantially simpler. More explicitly, one just needs to invoke the well-known fact that the minimal weak upper gradient of any Sobolev map to a metric space is zero on the preimage of a set of Hausdorff dimension zero. We omit a detailed consideration of this case for brevity.

        Let us start with the first part.
        
        We look at the locally finite Borel measure $\Lambda_{\Phi,p}$ on the separable, complete metric space $\Y$ in the context of \cref{theo:BLD}. By applying it, after fixing $\epsilon\in (0,1)$, we get the existence of Borel sets $Q_{i}\subseteq Q$, $i\in \Nat$, and of functions $H^i_{\infty}\in \LIP_1(\Y,\R^d_{\infty})$, $i\in \Nat$, with the corresponding properties. We also put
        \begin{equation}
            Q_0\coloneqq \bigg(Q\Big\backslash \bigcup\limits_{i\in \Nat}Q_i\bigg),
        \end{equation}
        so that we have $\Lambda_{\Phi,p}(Q_0)=0$. Fix $i\in \Nat$. Let $H^i_1\in \LIP(\Y,\R^d_1)$ be the map acting as $H^i_{\infty}$ and let $h^{i,1}$,..., $h^{i,d}\in \LIP_1(\Y)$ be the components of $H^i_{\infty}$. We are then exactly in a position to apply \cref{lem:FromMetrSpToUpGr} with respect to $Q_i$ and $H^i_1$, which gives us, in combination with \cref{prop:FinDimDir}, the following:
        \begin{equation}
            |\mathbf{D} \Phi|_p(x)\leq (1+\epsilon)\sum\limits_{l\in \overline{1,d}} \big|\mathbf{D} (h^{i,l}\circ \Phi)\big|_p(x)\qquad \text{for $\mX$-a.e. $x\in \Phi^{-1}(Q_{i})$}.
        \end{equation}
        This, in turn, implies, with the help of the property given in \eqref{eq:EnDefEq} and the Jensen inequality, that
        \begin{equation}
            \En_p[\Phi]\big(\Phi^{-1}(Q_{i})\big)\leq (1+\epsilon)^p d^{p-1}\sum\limits_{l\in \overline{1,d}} \En_p[h^{i,l}\circ \Phi]\big(\Phi^{-1}(Q_i)\big),
        \end{equation}
        from which we immediately get, with the use of \eqref{eq:PropCoMeas1}, that
        \begin{equation}
            \En_p[\Phi]\big(\Phi^{-1}(Q_i)\big)\leq (1+\epsilon)^p d^p \Theta_{\Phi,p}(Q_i).
        \end{equation}
        After summing these estimates over $i\in \Nat$ and using the property given in \eqref{eq:CoEnDefEq}, we can clearly write the following chain:
        \begin{equation}
        \begin{gathered}
            \Lambda_{\Phi,p}(Q)\leq \Lambda_{\Phi,p}(Q_0)+(1+\epsilon)^p d^p \sum\limits_{i\in \Nat}\Theta_{\Phi,p}(Q_i)\leq \\
            \leq (1+\epsilon)^p d^p \Theta_{\Phi,p}(Q).
            \end{gathered}
        \end{equation}
        Due to the arbitrariness of $\epsilon$, we arrive exactly at the desired estimate in \eqref{eq:EstEnCoMeas0}.

        We proceed with the second part.
        
        Based on the corresponding premise, we write, via the property given in \eqref{eq:CoEnDefEq} and the first part of the current theorem, that
        \begin{equation}
        \begin{gathered}
            \En_p[\Phi](\X)=\En_p[\Phi]\big(\X\backslash E\big)\leq \En_p[\Phi]\big(\Phi^{-1}(Q)\big)=\\
            =\Lambda_{\Phi,p}(Q)\leq  d^p\Theta_{\Phi,p}(\Y).
            \end{gathered}
        \end{equation}
        This is exactly the desired estimate in \eqref{eq:EstEnCoMeas00}.

        Thus, we can conclude the proof.
    \end{proof}
\end{theo}

\section{Proof finalization}

Here we provide the proof of our main result on the basis of what is obtained previously.

In this section, a metric measure space $\X=(\X,\dX,\mX)$, a metric space $\Y=(\Y,\dY)$, and an exponent $p\in [1,+\infty)$ are assumed fixed.

The theorem below forms a somewhat extended version of \cref{theo:MainRes}, within which, however, all the assumptions on the spaces are transferred onto a given map via the introduced notion of observable Hausdorff dimension. In turn, the corresponding connection links from the space conditions to the map condition will be verified in the lemmas afterwards. The forthcoming proof, as it is announced in \cref{ss:Intr}, goes with the idea developed within \cite[Theorem 25]{P26}. More explicitly, we are to consider an arbitrary ``finite-dimensional Lipschitz replica'' of the map and prove, via the estimate established in \cref{theo:EstEnCoMeas}, that its energy is bounded by the total mass of the measure given by \cref{lem:ExistsCoMeas}. Once this is done, the application of \cref{prop:EssSupInt} will give us the conclusion.
\begin{theo}\label{theo:GenTheo}
   Let $\Phi\in \mathrm{wS}^{1,p}(\X,\Y)$ be an $\mX$-essentially separably valued map of finite $p$-observable Hausdorff dimension. Then $\Phi\in \mathrm{S}^{1,p}(\X,\Y)$.
\begin{proof}
For a start, we deduce, by \cref{theo:Pettis}, that $\Phi\in \mathrm{M}(\X,\Y)$. Also, for brevity, let us put
\begin{equation}
    d\coloneqq \overline{\dim}^p_{\mathrm{oH}}(\Phi),
\end{equation}
so that, by the premise, we have $d<+\infty$.

We fix $J\in \Nat$ and $f^1,...,f^J\in \LIP_1(\Y)$.
    
   For each $j\in \overline{1,J}$, as $f^j\circ \Phi\in \mathrm{S}^{1,p}(\X)$, by definition, we can find $\psi^j\in [f^j\circ \Phi]_{\mX}\cap \mathrm{D}^{1,p}(\X)$. Let then $F\in \LIP(\Y,\R^J_{\infty})$ be the map given by the components $f^1,...,f^J$ and $\Psi$ be the map $\X\to \R^J_{\infty}$ given by the components $\psi^1,...,\psi^J$, so that the map $\Psi$ is an $\mX$-representative of $F\circ \Phi$ and, by \cref{prop:FinDimDir}, belongs moreover to $\mathrm{D}^{1,p}(\X,\R^J_{\infty})$. Thus, we can find an $\mX$-measurable set $E_0\subseteq \X$ with $\mX(E_0)=0$ such that
    \begin{equation}
        \Psi(x)=F\big(\Phi(x)\big)\qquad \text{for any $x\in \big(\X\backslash E_0\big)$}.
    \end{equation}
    Also, by \cref{defi:ObsDimMap}, we can find a Borel set $Q\subseteq \R^J_{\infty}$ with
    \begin{equation}
        \dim_{\mathrm{H}}(Q)\leq d
    \end{equation}
    such that
    \begin{equation}
        \Lambda_{\Psi,p}\big(\R^J_{\infty}\backslash Q\big)=\Lambda_{F\circ \Phi,p}\big(\R^J_{\infty}\backslash Q\big)=0.
    \end{equation}
    Putting
    \begin{equation}
        E\coloneqq E_0\cup (F\circ \Phi)^{-1}\big(\R^J_{\infty}\backslash Q\big),
    \end{equation}
    we get that
    \begin{equation}
        \Psi\big(\X\backslash E\big)\subseteq Q
    \end{equation}
    and, by the property given in \eqref{eq:CoEnDefEq}, that
    \begin{equation}
        \En_p[\Psi](E)=\En_p[F\circ \Phi](E)=0.
    \end{equation}
   We are exactly in a position to apply assertion II) of \cref{theo:EstEnCoMeas} with respect to $\Psi$, which gives us the following estimate:
\begin{equation}
    \En_p[\Psi](\X)\leq d^p \Theta_{\Psi,p}(\R^J_{\infty}).
\end{equation}
From the properties given in \eqref{eq:PropCoMeas2} and \eqref{eq:PropCoMeas3}, it then also immediately follows that
\begin{equation}
    \En_p[\Psi](\X)\leq d^p \Theta_{\Phi,p}(\Y).
\end{equation}
    This, in turn, can be transformed, thanks to assertion I) of \cref{prop:FinDimDir}, into the inequality
    \begin{equation}
        \int\limits_{\X} \bigvee\limits_{j\in \overline{1,J}}\Big(\big|\mathbf{D} (f^j\circ \Phi)\big|_p\Big)^p \D \mX\leq d^p\Theta_{\Phi,p}(\Y).
    \end{equation}

    We are left with the final step. Due to the arbitrariness of $J$ and $f^1,...,f^J$, what we obtained can be put in the following form:
    \begin{equation}
        \sup\limits_{\mathcal{F}}\int\limits_{\X} \bigvee\limits_{f\in \mathcal{F}} \Big(\big|\mathbf{D} (f\circ \Phi)\big|_p\Big)^p \D \mX\leq d^p \Theta_{\Phi,p}(\Y),
    \end{equation}
    where the supremum is taken over all nonempty finite families $\mathcal{F}\subseteq \LIP_1(\Y)$. In light of \cref{prop:EssSupInt}, we can write the chain
    \begin{equation}
    \begin{gathered}
    \int\limits_{\X} \Bigg(\bigvee\limits_{f\in \LIP_1(\Y)} \big|\mathbf{D} (f\circ \Phi)\big|_p\Bigg)^p \D \mX=\\=\int\limits_{\X} \bigvee\limits_{f\in \LIP_1(\Y)} \Big(\big|\mathbf{D} (f\circ \Phi)\big|_p\Big)^p \D \mX\leq d^p\Theta_{\Phi,p}(\Y),
    \end{gathered}
    \end{equation}
    which clearly implies that
    \begin{equation}
        \bigvee\limits_{f\in \LIP_1(\Y)} \big|\mathbf{D} (f\circ \Phi)\big|_p\in \Leb_+^p(\X).
    \end{equation}
    This, by definition, ensures that $\Phi\in \mathrm{S}^{1,p}(\X,\Y)$, as desired. This completes the proof.
\end{proof}
\end{theo}

Now we check that condition I) from \cref{theo:MainRes} entails the suitable dimensional bound on a given map as in \cref{theo:GenTheo}.  This will imply the first constituent of \cref{theo:MainRes}.
\begin{lem}\label{lem:BoundsSource}
    Suppose the space $\X$ is of finite Hausdorff dimension and satisfies the $p$-Sobolev-to-Luzin--Lipschitz property. Let $\Phi\in \mathrm{wS}^{1,p}(\X,\Y)$. Then the map $\Phi$ is of finite $p$-observable Hausdorff dimension with
    \begin{equation}
        \overline{\dim}_{\mathrm{oH}}^p(\Phi)\leq \overline{\dim}_{\mathrm{H}}(\X).
    \end{equation}
    \begin{proof}

    For brevity, let us put
    \begin{equation}
        d\coloneqq \overline{\dim}_{\mathrm{H}}(\X),
    \end{equation}
    so that we have $d<+\infty$ by the premise.
    
        Fix $J\in \Nat$ and $F\in \LIP(\Y,\R^J_{\infty})$. Since the inclusion $F\circ \Phi\in \mathrm{S}^{1,p}(\X,\R^J_{\infty})$ holds, by the properties given in \eqref{eq:wSobLipSt} and \eqref{eq:FinDimBanwSobSob}, and the space $\X$ satisfies the $p$-Sobolev-to-Luzin--Lipschitz property, we can find compact sets $K_i\subseteq \X$, $i \in \Nat$, with
        \begin{equation}\label{eq:BoundsSource1}
            \mX\bigg(\X\Big\backslash \bigcup\limits_{i\in \Nat}K_i\bigg)=0
        \end{equation}
        such that
        \begin{equation}
            \text{the function $(F\circ \Phi)|_{K_i}$ is Lipschitz for each $i\in \Nat$}.
        \end{equation}
        Putting then
        \begin{equation}
            Q\coloneqq \bigcup\limits_{i\in \Nat}F\big(\Phi(K_i)\big),
        \end{equation}
        we see that the set $Q$ is Borel with
        \begin{equation}
            \dim_{\mathrm{H}}(Q)\leq d.
        \end{equation}
        We now notice that
\begin{equation}
    \big(\X\backslash (F\circ \Phi)^{-1}(Q)\big)\subseteq \bigg(\X \Big\backslash \bigcup\limits_{i\in \Nat}K_i\bigg),
\end{equation}
        whence, due to \eqref{eq:CoEnDefEq} and \eqref{eq:BoundsSource1}, we get that
        \begin{equation}
           \Lambda_{F\circ \Phi,p}\big(\R^J_{\infty}\backslash Q\big)= \En_p[F\circ \Phi]\big(\X\backslash (F\circ \Phi)^{-1}(Q)\big)=0.
        \end{equation}
        Due to the arbitrariness of $J$ and $F$, we immediately obtain, by \cref{defi:ObsDimMap}, that
        \begin{equation}
            \overline{\dim}_{\mathrm{oH}}^p(\Phi)\leq d,
        \end{equation}
        as desired. This finishes the proof.
    \end{proof}
\end{lem}

And now we verify that condition II) from \cref{theo:MainRes} yields the necessary dimensional bound on a given map in the sense of \cref{theo:GenTheo} as well. This will imply the second constituent of \cref{theo:MainRes}.
\begin{lem}\label{lem:BoundsTarget}
    Suppose the space $\Y$ is of finite Hausdorff dimension. Let $\Phi\in \mathrm{wS}^{1,p}(\X,\Y)$. Then the map $\Phi$ is of finite $p$-observable Hausdorff dimension with
    \begin{equation}
        \overline{\dim}_{\mathrm{oH}}^p(\Phi)\leq \overline{\dim}_{\mathrm{H}}(\Y).
    \end{equation}
    \begin{proof}
    For brevity, let us put
    \begin{equation}
        d\coloneqq \overline{\dim}_{\mathrm{H}}(\Y),
    \end{equation}
    so that $d<+\infty$ by the premise.
    
        Fix $J\in \Nat$ and $F\in \LIP(\Y,\R^J_{\infty})$. Since
        \begin{equation}
            \dim_{\mathrm{H}}\big(F(\Y)\big)\leq d
        \end{equation}
        by the premise, we can find a Borel set $Q\subseteq \R^J_{\infty}$ with $F(\Y)\subseteq Q$ such that
        \begin{equation}
            \dim_{\mathrm{H}}(Q)\leq d.
        \end{equation}
       We also observe that $F\circ \Phi\in \mathrm{S}^{1,p}(\X,\R^J_{\infty})$, as follows from the properties given in \eqref{eq:wSobLipSt} and \eqref{eq:FinDimBanwSobSob}. All this immediately implies that
        \begin{equation}
               \Lambda_{F\circ \Phi,p}\big(\R^J_{\infty}\backslash Q\big)=0.
        \end{equation}
        Thus, due to the arbitrariness of $J$ and $F$, we immediately obtain, by \cref{defi:ObsDimMap}, that
        \begin{equation}
            \overline{\dim}^p_{\mathrm{oH}}(\Phi)\leq d,
        \end{equation}
        as desired. This finishes the proof.
    \end{proof}
\end{lem}

{\bf Acknowledgments.} The study continues research begun by the author as part of his Master's Thesis at the Moscow Institute of Physics and Technology. The author sincerely appreciates all the fruitful discussions on the topic with A. I. Tyulenev having taken place over the recent years, as well as the motivation provided by him for creating this paper.

\printbibliography

@article{AGS14,
  author  = {Ambrosio, Luigi and Gigli, Nicola and
             Savar{\'e}, Giuseppe},
  title   = {Calculus and heat flow in metric measure spaces
             and applications to spaces with {Ricci} bounds
             from below},
  journal = {Inventiones Mathematicae},
  volume  = {195},
  number  = {2},
  pages   = {289--391},
  year    = {2014},
  doi     = {10.1007/s00222-013-0456-1}
}

@article{SV25,
  author  = {Vodopyanov, Sergey K. and Sboev, Danil A.},
  title   = {Composition operators on Sobolev spaces
             with domain of metric measure spaces. II},
  journal = {Matematicheskii Sbornik},
  volume  = {217},
  number  = {8},
  pages   = {87--129},
  year    = {2026},
  doi     = {10.4213/sm10380}
}

@article{O26.2,
  title={On one relaxation of the bounded-length-distortion condition in the context of metric measure spaces},
  author={Oleinik, Roman D},
  journal={arXiv preprint arXiv:2607.00278},
  year={2026}
}

@article{ABT18,
  title={Lusin-type approximation of Sobolev by Lipschitz functions, in Gaussian and $\mathrm{RCD}(K,\infty)$ spaces},
  author={Ambrosio, Luigi and Bru{\`e}, Elia and Trevisan, Dario},
  journal={Advances in Mathematics},
  volume={339},
  pages={426--452},
  year={2018},
  publisher={Elsevier}
}

@article{VU02,
  author  = {Vodop'yanov, S. K. and Ukhlov, A. D.-O.},
  title   = {Superposition operators in Sobolev spaces},
  journal = {Russian Mathematics (Iz. VUZ)},
  volume  = {46},
  number  = {10},
  year    = {2002},
  pages   = {9--31},
}

@article{ST26,
  author  = {Stolyarov, Dmitriy and Tyulenev, Alexander},
  title   = {Catching jumps of metric-valued mappings
             with Lipschitz functions},
  journal = {Annali della Scuola Normale Superiore di Pisa,
             Classe di Scienze},
  year    = {2026},
  note    = {To appear},
  doi     = {10.2422/2036-2145.202603_007}
}

@article{AHPCS18,
  title={Infinity modulus and the essential metric},
  author={Albin, Nathan and Hoppis, Jared and Poggi-Corradini, Pietro and Shanmugalingam, Nageswari},
  journal={Journal of Mathematical Analysis and Applications},
  volume={467},
  number={1},
  pages={570--584},
  year={2018},
  publisher={Elsevier}
}

@article{T26,
  title={A New insight into the Ambrosio--Reshetnyak approach to Sobolev maps},
  author={Tyulenev, Alexander},
  journal={arXiv preprint arXiv:2610.04384},
  year={2026}
}

@unpublished{P26,
  author   = {Pavlov, Stepan V.},
  title    = {Estimate of the differential of a mapping from the plane into a metric space by the differential of its composition with a $1$-Lipschitz function},
  note     = {Manuscript in preparation (2026)},
  sortyear = {2026}
}

@article{S26,
  title={A random Lipschitz function detecting jumps},
  author={Stolyarov, Dmitriy},
  journal={arXiv preprint arXiv:2608.09868},
  year={2026}
}

@article{B26,
  author    = {V. I. Bakhtin},
  title     = {Criterion for absolute continuity of curves in metric spaces},
  journal   = {Mat. Zametki},
  volume    = {120},
  number    = {3},
  pages     = {339--347},
  year      = {2026},
  doi       = {10.4213/mzm14855},
  url       = {https://www.mathnet.ru/eng/mzm14855}
}

@article{O25,
  author  = {Oleinik, Roman},
  title   = {Asymptotic Relations of the
             Bourgain--Brezis--Mironescu Type for
             Mappings Between Singular Spaces},
  journal = {The Journal of Geometric Analysis},
  volume  = {35},
  number  = {7},
  eid     = {196},
  year    = {2025},
  doi     = {10.1007/s12220-025-02019-y}
}

@article{O26.1,
  author  = {Oleinik, Roman D.},
  title   = {On estimates of the Bourgain--Brezis--Mironescu
             type in the singular context},
  journal = {Journal of Mathematical Analysis and Applications},
  volume  = {564},
  number  = {2},
  eid     = {130887},
  year    = {2026},
  doi     = {10.1016/j.jmaa.2026.130887}
}

@article{LD13,
  title={Lipschitz and path isometric embeddings of metric spaces},
  author={Le Donne, Enrico},
  journal={Geometriae Dedicata},
  volume={166},
  number={1},
  pages={47--66},
  year={2013},
  publisher={Springer}
}

@article{OT24,
  author  = {Oleinik, Roman D. and Tyulenev, Alexander I.},
  title   = {Characterization of {AC} and Sobolev curves
             via Lipschitz post-compositions},
  journal = {Mathematische Annalen},
  volume  = {395},
  eid     = {104},
  year    = {2026},
  doi     = {10.1007/s00208-026-03516-3}
}

@book{NB10,
  title={Topological vector spaces},
  author={Narici, Lawrence and Beckenstein, Edward},
  edition={2},
  year={2010},
  publisher={Chapman and Hall/CRC},
  isbn={9781584888666}
}

@article{AILP24,
  title={Metric Sobolev Spaces I: Equivalence of Definitions},
  author={Ambrosio, Luigi and Ikonen, Toni and Lu{\v{c}}i{\'c}, Danka and Pasqualetto, Enrico},
  journal={Milan Journal of Mathematics},
  volume={92},
  number={2},
  pages={255--347},
  year={2024},
  publisher={Springer}
}

@article{BNP25,
  author  = {Brena, Camillo and Nobili, Francesco and Pasqualetto, Enrico},
  title   = {Maps of bounded variation from {PI} spaces to metric spaces},
  journal = {Annali della Scuola Normale Superiore di Pisa, Classe di Scienze},
  series  = {5},
  volume  = {26},
  year    = {2025},
  pages   = {1475--1519},
  doi     = {10.2422/2036-2145.202307_003},
}

@book{W18,
  title={Lipschitz algebras},
  author={Weaver, Nik},
  edition={2},
  year={2018},
  publisher={World Scientific},
  isbn={9789814740630}
}

@article{AGS13,
  author  = {Ambrosio, Luigi and Gigli, Nicola and
             Savar{\'e}, Giuseppe},
  title   = {Density of Lipschitz functions and equivalence
             of weak gradients in metric measure spaces},
  journal = {Revista Matem{\'a}tica Iberoamericana},
  volume  = {29},
  number  = {3},
  pages   = {969--996},
  year    = {2013},
  doi     = {10.4171/RMI/746}
}

@book{GP20,
Author = {Gigli, Nicola and Pasqualetto, Enrico},
ISBN = {978-3-030-38612-2},
Publisher = {Springer, Cham},
Series = {SISSA Springer Series, 2.},
Title = {Lectures on nonsmooth differential geometry.},
URL = {https://search.ebscohost.com/login.aspx?direct=true&db=msn&AN=MR4321459&lang=ru&site=ehost-live},
Year = {2020},
}

@article{GT21,
  author  = {Gigli, Nicola and Tyulenev, Alexander},
  title   = {Korevaar--{Schoen}'s energy on strongly
             rectifiable spaces},
  journal = {Calculus of Variations and Partial Differential Equations},
  volume  = {60},
  number  = {6},
  eid     = {235},
  year    = {2021},
  doi     = {10.1007/s00526-021-02028-z}
}

@book{HKST15,
  author    = {Heinonen, Juha and Koskela, Pekka and Shanmugalingam, Nageswari and Tyson, Jeremy T.},
  title     = {Sobolev spaces on metric measure spaces: An approach based on upper gradients},
  series    = {New Mathematical Monographs},
  volume    = {27},
  publisher = {Cambridge University Press, Cambridge},
  year      = {2015},
  pages     = {xii+434},
  isbn      = {978-1-107-09234-1},
  doi       = {10.1017/CBO9781316135914},
  url       = {https://doi.org/10.1017/CBO9781316135914}
}

@article {R97,
    AUTHOR = {Reshetnyak, Yu. G.},
     TITLE = {{S}obolev classes of functions with values in a metric space},
   JOURNAL = {Sibirsk. Mat. Zh.},
  FJOURNAL = {Rossi\v{\i}skaya Akademiya Nauk. Sibirskoe Otdelenie. Institut
              Matematiki im. S. L. Soboleva. Sibirski\v{\i} Matematicheski\v{\i} Zhurnal},
    VOLUME = {38},
      YEAR = {1997},
    NUMBER = {3},
     PAGES = {657--675, iii--iv},
      ISSN = {0037-4474},
   MRCLASS = {46E35},
  MRNUMBER = {1457485},
MRREVIEWER = {H. Triebel},
       DOI = {10.1007/BF02683844},
       URL = {https://doi.org/10.1007/BF02683844},
}

@article{BB18,
  author  = {Bj{\"o}rn, Anders and Bj{\"o}rn, Jana},
  title   = {Local and semilocal Poincar{\'e} inequalities
             on metric spaces},
  journal = {Journal de Math{\'e}matiques Pures et Appliqu{\'e}es},
  volume  = {119},
  pages   = {158--192},
  year    = {2018},
  doi     = {10.1016/j.matpur.2018.05.005}
}

@article{A90,
  author  = {Ambrosio, Luigi},
  title   = {Metric space valued functions of bounded variation},
  journal = {Annali della Scuola Normale Superiore di Pisa.
             Classe di Scienze. Serie IV},
  volume  = {17},
  number  = {3},
  pages   = {439--478},
  year    = {1990}
}
\end{document}